\documentclass[leqno,10pt]{amsart}
\usepackage{amssymb,amsmath,amsfonts,amsthm,mathtools,latexsym,ifthen,bezier}\usepackage{graphicx}
\usepackage{amsxtra}
\usepackage{xspace}
\usepackage{url}
\usepackage[english]{babel} 

\usepackage{etoolbox}

\newcommand{\kla}[1]{ {\langle #1 \rangle} }

\newcommand{\st}{\;|\;}

\newcommand{\ran}{ {\rm ran} }

\newcommand{\sub}{\subseteq}

\newfont{\ssi}{cmssi12 at 12pt}

\newcommand{\rest}{{\restriction}}
\newcommand{\cf}{ {\rm cf} }

\newcommand{\verl}{{{}^\frown}}
\newcommand{\leer}{\emptyset}
\newcommand{\ohne}{\setminus}

\newcommand{\id}{ {\rm id} }

\newenvironment{ea*}{\begin{eqnarray*}}{\end{eqnarray*}}

\newcounter{claimnumber}
\AtBeginEnvironment{thm}{\setcounter{claimnumber}{0}}
\AtBeginEnvironment{lem}{\setcounter{claimnumber}{0}}
\AtBeginEnvironment{cor}{\setcounter{claimnumber}{0}}
\AtBeginEnvironment{obs}{\setcounter{claimnumber}{0}}
\AtBeginEnvironment{fact}{\setcounter{claimnumber}{0}}
\newenvironment{NumberedClaim}%
{\begin{enumerate}[label=(\arabic*)]
\refstepcounter{claimnumber}
\setcounter{enumi}{\value{claimnumber}-1}
\item}%
{\end{enumerate}}

\setbox0=\hbox{$\longrightarrow$}
\newcommand{\To}{\longrightarrow}

\newcommand{\power}{{\mathcal{P}}}

\newcommand{\calT}{\mathcal{T}}

\newcommand{\ba}{{\bar{a}}}

\newcommand{\bA}{{\bar{A}}}

\newcommand{\bD}{{\bar{D}}}

\newcommand{\bG}{{\bar{G}}}
\newcommand{\bH}{{\bar{H}}}

\newcommand{\bM}{{\bar{M}}}
\newcommand{\bP}{{\bar{\P}}}

\newcommand{\bS}{{\bar{S}}}
\newcommand{\bT}{{\bar{T}}}

\newcommand{\bX}{{\bar{X}}}

\newcommand{\balpha}{{\bar{\alpha}}}

\newcommand{\bkappa}{{\bar{\kappa}}}

\newcommand{\blambda}{{\bar{\lambda}}}

\newcommand{\bp}{{\bar{p}}}

\newcommand{\bN}{{\bar{N}}}

\newcommand{\vx}{{\vec{x}}}
\newcommand{\vy}{{\vec{y}}}

\newcommand{\vD}{{\vec{D}}}

\newcommand{\vM}{{\vec{M}}}

\newcommand{\vS}{{\vec{S}}}
\newcommand{\vT}{{\vec{T}}}

\newcommand{\seq}[2]{{\langle#1\;|\;}\linebreak[0]{#2\rangle}}

\renewcommand{\phi}{\varphi}
\newcommand{\card}[1]{\overline{\overline{#1}}}

\newcommand{\ZFC}{\ensuremath{\mathsf{ZFC}}\xspace}

\newcommand{\ZFCm}{\ensuremath{{\ZFC}^-}\xspace}

\newcommand{\V}{\ensuremath{\mathrm{V}}\xspace}

\newcommand{\forces}{\Vdash}

\def\<#1>{\langle#1\rangle}
\newcommand{\restrict}{\upharpoonright}

\renewcommand{\P}{{\mathord{\mathbb P}}}
\newcommand{\Q}{{\mathord{\mathbb Q}}}
\newcommand{\R}{{\mathord{\mathbb R}}}

\newcommand{\SC}{\ensuremath{\mathsf{SC}}\xspace}

\newcommand{\MA}{\ensuremath{\mathsf{MA}}}
\newcommand{\MM}{\ensuremath{\mathsf{MM}}\xspace}

\newcommand{\SPFA}{\ensuremath{\mathsf{SPFA}}\xspace}

\newcommand{\FA}{\ensuremath{\mathsf{FA}}}
\newcommand{\SCFA}{\ensuremath{\mathsf{SCFA}}\xspace}

\newcommand{\MP}{\ensuremath{\mathsf{MP}}}

\newcommand{\ColNothing}{\mathrm{Col}}
\newcommand{\Col}[1]{\ColNothing(#1)}

\newcommand{\MPColNothing}[1]{\MP_{\Col{\dot{\kappa}}}}

\newcommand{\isomorphic}{\cong}

\newcommand{\CH}{\ensuremath{\mathsf{CH}}\xspace}

\newcommand{\reals}{{\mathord{\mathbb{R}}}}

\newcommand{\A}{\mathcal{A}}

\newcommand{\Add}{\mathord{\mathrm{Add}}}

\newcommand{\lift}{\mathop{\uparrow}}

\newcommand{\NS}{\mathsf{NS}}

\newcommand{\Refl}{\ensuremath{\mathsf{Refl}}}

\newcommand{\OSR}[1]{\ensuremath{\mathsf{OSR}_{#1}}}
\newcommand{\DSR}[1]{\ensuremath{\mathsf{DSR}{(#1)}}}
\newcommand{\SRP}{\ensuremath{\mathsf{SRP}}\xspace}

\newcommand{\Todorcevic}{Todor\v{c}evi\'{c}\xspace}

\newtheorem{thm}{Theorem}[section]
\newtheorem*{thm*}{Theorem} 
\newtheorem{cor}[thm]{Corollary}
\newtheorem{lem}[thm]{Lemma}
\newtheorem{obs}[thm]{Observation}

\newtheorem{fact}[thm]{Fact}

\theoremstyle{definition}
\newtheorem{defn}[thm]{Definition}
\newtheorem{question}[thm]{Question}

\newtheorem{notation}[thm]{Notation}

\newcommand{\thistheoremname}{}
\newtheorem{genericthm}[thm]{\thistheoremname}

\theoremstyle{remark}
\newtheorem{remark}[thm]{Remark}

\usepackage{enumitem}
\renewcommand{\A}{\mathbb{A}}
\renewcommand{\card}[1]{\mathsf{card}(#1)}

\newcommand{\U}{\mathbb{U}}
\newcommand{\T}{\mathbb{T}}
\newcommand{\infSC}{\ensuremath{\infty\text{-}\mathsf{SC}}\xspace}

\newcommand{\infSCFA}{\ensuremath{\infty\text{-}\mathsf{SCFA}}\xspace}
\renewcommand{\DSR}{\mathsf{DSR}}
\renewcommand{\OSR}{\ensuremath{\mathsf{OSR}}\xspace}
\newcommand{\uDSR}{\mathsf{uDSR}}
\newcommand{\lifting}{\mathsf{lifting}}

\newcommand{\sDSR}{\ensuremath{\mathsf{sDSR}}}
\newcommand{\refl}{\ensuremath{\mathsf{Refl}}}
\newcommand{\projectdown}{\mathbin{\downarrow}}
\newcommand{\projectup}{\mathbin{\uparrow}}
\newcommand{\calS}{\mathcal{S}}

\newcommand{\calN}{\mathcal{N}}
\newcommand{\bcalN}{\bar{\mathcal{N}}}
\newcommand{\calL}{\mathcal{L}}

\newcommand{\bb}[1]{\mathbb{#1}}

\newcommand{\SSP}{\ensuremath{\mathsf{SSP}}}

\def\undertilde#1{\mathord{\vtop{\ialign{##\crcr
$\hfil\displaystyle{#1}\hfil$\crcr\noalign{\kern1.5pt\nointerlineskip}
$\hfil\tilde{}\hfil$\crcr\noalign{\kern1.5pt}}}}}

\begin{document}
\title{Fragments of \SRP, the nonstationary ideal, and diagonal reflection}
\author[Fuchs]{Gunter Fuchs}
\address[G.~Fuchs]{Mathematics Program,
          The Graduate Center of The City University of New York,
          365 Fifth Avenue, New York, NY 10016
          \&
          Department of Mathematics,
          College of Staten Island of CUNY, 2800 Victory Boulevard,
          Staten Island, NY 10314}
\email{Gunter.Fuchs@csi.cuny.edu}
\urladdr{http://www.math.csi.cuny/edu/\textasciitilde{}fuchs}

\author[Sakai]{Hiroshi Sakai}
\address[H.~Sakai]{The Graduate School of Mathematical Sciences,
	3-8-1 Komaba Meguro-ku, Tokyo 153-8914, Japan}
\email{hsakai@ms.u-tokyo.ac.jp}

\subjclass[2020]{03E35, 03E50, 03E57}
\keywords{Forcing axioms, strong reflection principle, subcomplete forcing, saturation of the nonstationary ideal}

\date{\today}

\thanks{The first author was supported by PSC-CUNY award \# 68379-00 56 and Simons collaboration grant 580600. The second author is supported by JSPS KAKENHI Grant Number 24K06828. \\
\indent This manuscript was submitted to the Journal of Symbolic Logic.}

\begin{abstract}
We investigate the relationships between fragments of \SRP and diagonal reflection principles. It turns out that the subcomplete fragment of \SRP, augmented with \CH, implies that the nonstationary ideal on $\omega_1$ is not saturated, but that Larson's diagonal reflection principle $\OSR_{\omega_2}$ holds. In contrast, it was known that the full principle \SRP implies the saturation of the nonstationary ideal on $\omega_1$, but it does not imply $\OSR_{\omega_2}$. We find several instances of this phenomenon: pairing (variations of) the subcomplete fragment of \SRP with stronger failures of $\NS_{\omega_1}$ to be saturated allows us to conclude more instances of diagonal reflection principles, while boosting (the corresponding variations of) the full principle \SRP with strong forms of the saturation of $\NS_{\omega_1}$ allows us to produce models where these principles still hold, yet weaker versions of diagonal reflection fail.
\end{abstract}

\maketitle

\section{Introduction}

Soon after the introduction of Martin's Maximum \MM, the strong reflection principle \SRP was introduced as a purely combinatorial principle that encapsulates many of the consequences of \MM. One such consequence of \MM is that $\NS_{\omega_1}$, the nonstationary ideal on $\omega_1$,  is $\omega_2$-saturated. Isolating consequences of \MM that do not follow from \SRP became a focal question, and many of these concern diagonal stationary reflection principles. In this direction, Larson showed that one such consequence is the principle $\OSR_{\omega_2}$: it follows from \MM but not from \SRP.

If we weaken \SRP to its subcomplete fragment, denoted $\infSC$-\SRP, then the resulting principle is consistent with the nonstationary ideal on $\omega_1$ failing to be saturated, \CH, and even $\diamondsuit$. 

In this article, we investigate the effects of (variations of) $\infSC$-\SRP, together with more or less strong failures of $\NS_{\omega_1}$ being $\omega_2$-saturated, on the extent of (various forms) of diagonal stationary reflection principles. 

The initial result in this direction is Theorem \ref{thm:SC-SRP+CHimpliesOSR}, stating that the subcomplete fragment of \SRP, together with \CH, implies Larson's $\OSR_{\omega_2}$. Here, adding \CH to $\infSC$-\SRP is a natural thing to do, as is explained in great detail in \cite{Fuchs:Errata}, but also, in the context of $\infSC$-\SRP, \CH implies that $\NS_{\omega_1}$ is not $\omega_2$-saturated, as we point out in Observation \ref{obs:Woodin}. This result is surprising, because we weaken the assumption of \SRP to just its subcomplete fragment, and we add an assumption which appears to weaken it even further by taking away the seemingly strong consequence of the saturatedness of $\NS_{\omega_1}$, but as a consequence, we obtain a principle of diagonal stationary reflection that is out of reach for \SRP.   

Exploring this phenomenon further, we consider a natural forcing axiom type strengthening of this fragment of \SRP which we call \emph{enhanced} $\infSC$-$\SRP^\FA$, and add a stronger failure of the saturatedness of $\NS_{\omega_1}$, namely Jensen's $\diamondsuit$ principle, to obtain in Theorem \ref{thm:AddingDiamondGivesDiagonalReflection} as a consequence a generalization of $\OSR_{\omega_2}$ to all larger regular cardinals greater than $\omega_1$.

We consider the results mentioned up to now to be positive results, in the sense that they state that certain versions of \SRP, together with appropriate assumptions on the nonsaturatedness of $\NS_{\omega_1}$ \emph{do imply} certain types of diagonal reflection. In Section \ref{sec:Separations}, we deal with negative results, i.e., results showing that certain assumptions of the previous kind do \emph{not imply} various diagonal stationary reflection principles. 
We begin by generalizing methods of Larson \cite{Larson:SeparatingSRP} and \cite{Fuchs-LambieHanson:SeparatingDSR} to separate ``forcing axiom versions'' of \SRP from diagonal stationary reflection principles, first at $\omega_2$, then at larger regular cardinals, and finally we prove a result showing that the enhanced forcing axiom version of \SRP does not imply a weak diagonal reflection principle at uncountable $\lambda>\omega_2$ (the point being that the enhanced forcing axiom version of the subcomplete fragment of \SRP, together with $\diamondsuit$, implies even stronger forms of diagonal reflection at all regular cardinals greater than $\omega_1$.)  
Moving in the opposite direction, in Subsection \ref{subsec:NSomega1-dense}, we show that if we enhance \SRP with the $\omega_1$-density of the nonstationary ideal on $\omega_1$, i.e., \emph{strengthening} its saturation properties, then an even weaker principle than $\OSR_{\omega_2}$ can be forced to fail, while \SRP continues to hold. This result was known by work of Fuchs \& Lambie-Hanson \cite{Fuchs-LambieHanson:SeparatingDSR} - however, back then, the assumption was not known to be consistent. We observe here that it is, and we fix a slight imprecision in the proof. Moreover, we observe that these separation results extend to a natural forcing axiom version of \SRP; see Theorem \ref{thm:INSdenseuDSRfailsAtOmega2} for the exact formulation of the result.

Thus, the picture that emerges is that in the context of fragments of \SRP, failures of strong saturation properties of the nonstationary ideal on $\omega_1$, apparently forms of weakness, actually boost the available diagonal reflection principles.

\section{Reflection principles}

In this section, we summarize the different reflection principles we will consider.

\subsection{Simultaneous, diagonal and exact reflection of stationary sets of ordinals}

The reflection principles formulated in the following definition were originally introduced in Fuchs \cite{Fuchs:DiagonalReflection} and Fuchs \& Lambie-Hanson \cite{Fuchs-LambieHanson:SeparatingDSR}. The concept of an \emph{exact} reflection point is from Fuchs \cite{Fuchs:CanonicalFragmentsOfSRP}. Stationary reflection is of course a staple of set-theoretic research, and a good starting point is Cummings, Foreman \& Magidor \cite{SquaresScalesStationaryReflection}.

\begin{defn}
\label{defn:DiagonalReflection}
Let $\lambda$ be a regular cardinal. If $S\sub\lambda$ is stationary in $\lambda$, then an ordinal $\rho<\lambda$ of uncountable cofinality is called a \emph{reflection point} of $S$ if $S\cap\rho$ is stationary in $\rho$. If $\calS$ is a set of stationary subsets of $\lambda$, then an ordinal $\rho$ is a \emph{simultaneous reflection point} of $\calS$ if it is a reflection point of every $S\in\calS$. It is an \emph{exact reflection point} of $\calS$ if $\rho\cap\bigcup{\calS}$ contains a club (equivalently, $\rho$ is not a reflection point of $\lambda\ohne\bigcup\calS$.) Further, it is a \emph{strongly exact reflection point} of $\calS$ if for every $T\sub\rho$ which is stationary in $\rho$, there is an $S\in\calS$ with $T\cap S$ stationary in $\rho$. Clearly, a strongly exact reflection point is an exact reflection point.

Let $S$ be a stationary subset of some regular cardinal $\lambda$, and let $\kappa$ be a cardinal.

The \emph{simultaneous reflection principle} $\refl(\kappa,S)$ says that whenever $\calS$ is a set of $\kappa$ many stationary subsets of $S$, then $\calS$ has a simultaneous reflection point. The \emph{(strongly) exact simultaneous reflection principle,} denoted \emph{(strongly) exact} $\refl(\kappa,S)$, postulates the existence of a (strongly) exact simultaneous reflection point in this situation.

The \emph{diagonal reflection principle} $\DSR({<}\kappa,S)$ says that if $\langle S_{\alpha,\xi}\st\alpha<\lambda,\ \xi<\zeta_\alpha\rangle$ is a matrix of stationary subsets of $S$, where $\zeta_\alpha<\kappa$ for every $\alpha<\lambda$, then there are an ordinal $\rho<\lambda$ of uncountable cofinality and a set $R\subseteq\rho$ which contains a club in $\rho$, such that $\rho$ is a simultaneous reflection point of the set $\vec{S}|R=\{S_{\alpha,\xi}\st\alpha\in R,\ \xi<\zeta_\alpha\}$. The \emph{(strongly) exact diagonal reflection principle}, denoted \emph{(strongly) exact} $\DSR({<}\kappa,S)$, further requires that $\rho$ is a (strongly) exact reflection point of $\vec{S}|R$.

The versions of these principles in which $j_\alpha\le\kappa$ are denoted $\DSR(\kappa,S)$, (strongly) exact $\DSR(\kappa,S)$, respectively.

If $R$ is only required to be \emph{unbounded} in $\rho$, then the resulting principles are called $\uDSR({<}\kappa,S)$, (strongly) exact $\uDSR({<}\kappa,S)$,
and if it is required to be \emph{stationary,} then they are denoted $\sDSR({<}\kappa,S)$, (strongly) exact $\sDSR({<}\kappa,S)$, respectively.
\end{defn}

Diagonal reflection principles were introduced in \cite{Fuchs:DiagonalReflection} mainly because of their connection to square principles (introduced originally by Jensen \cite{FS}). Namely, if $\lambda$ is regular and $\kappa<\lambda$ is a cardinal, then $\DSR({<}\kappa,S)$, for any stationary $S\sub\lambda$, implies the failure of $\Box(\lambda,{<}\kappa)$. The assumption was subsequently reduced to $\sDSR({<}\kappa,S)$ in \cite{Fuchs-LambieHanson:SeparatingDSR}.

The main import of \emph{exact} reflection is the following fact. The case most often encountered is that $\kappa=\omega_1$ here; see \cite{FMS:MM1}. 

\begin{fact}[Fuchs {\cite[Fact 3.15]{Fuchs:CanonicalFragmentsOfSRP}}] 
\label{fact:PointOfExactness}
If $\lambda>\kappa$ are regular cardinals and $S\sub\lambda$ is stationary in $\lambda$ such that exact $\Refl(\kappa,S)$ holds, then $\lambda^\kappa=\lambda$.
\end{fact}

The diagonal reflection principles can be viewed as different ways of generalizing a reflection principle due to Paul Larson. 

\begin{notation}
For a regular cardinal $\kappa$ and a limit ordinal $\lambda>\kappa$, we write
\[S^\lambda_\kappa=\{\alpha<\lambda\st\cf(\alpha)=\kappa\}.\]	
\end{notation}

\begin{defn}[Larson {\cite{Larson:SeparatingSRP}}]
\label{def:OSR}
The principle $\OSR_{\omega_2}$ (for ``ordinal stationary reflection'') states that whenever $\vS=\seq{S_\alpha}{\alpha<\omega_2}$ is a sequence of stationary subsets of $S^{\omega_2}_\omega$, there is an ordinal $\rho<\omega_2$ of uncountable cofinality which is a simultaneous reflection point of $\{S_\alpha\st \alpha<\rho\}$. The principle \emph{(strongly) exact} $\OSR_{\omega_2}$ requires $\rho$ to be a (strongly) exact reflection point in this situation.
\end{defn}

It was shown in \cite[Lemma 2.4]{Fuchs-LambieHanson:SeparatingDSR} that $\OSR_{\omega_2}$ is equivalent to $\uDSR(\omega_1,S^{\omega_2}_\omega)$.

\subsection{Fragments of the strong reflection principle}

Before reviewing how to make sense of the fragment of the strong reflection principle associated to a forcing class, a method developed in \cite{Fuchs:CanonicalFragmentsOfSRP}, we fix some notation around generalized stationarity. See \cite[\S4]{Jech:StationarySetsHST} for the basic theory of stationary subsets of $\power_\kappa(A)$. We will mostly be dealing with stationary subsets of $[A]^\omega$.

\begin{defn}
\label{defn:ProjectionsAndLiftings}
Let $\kappa$ be a regular cardinal, and let $S\sub\kappa$ be unbounded. Let $\kappa\sub A$. Then
\[\lifting(S,[A]^\omega)=\{x\in[A]^\omega\st\sup(x\cap\kappa)\in S \}\]
is the \emph{lifting} of $S$ to $[A]^\omega$.

Now suppose $\omega_1 \subseteq B_0 \subseteq A \subseteq B_1$ and $X\sub[A]^\omega$ is stationary. Then we define the projections of $X$ to $[B_0]^\omega$ and $[B_1]^\omega$ by
\begin{gather*}
X\projectdown[B_0]^\omega=\{x\cap B_0 \st x\in X\} , \\
X\projectup[B_1]^\omega=\{y\in[B_1]^\omega\st y\cap A\in X\} .
\end{gather*}
\end{defn}

The strong reflection principle was introduced by \Todorcevic in 1987 (unpublished, but see \cite[p.~57]{Bekkali:TopicsInST}), and its formulation was different, but equivalent to the definition given here, which follows Feng \& Jech \cite{FengJech:ProjectiveStationarityAndSRP}.

\begin{defn}
\label{def:P_S}
Let $H$ be an uncountable set, and let $X\sub[H]^\omega$. Then
$\P_X$ is the forcing notion consisting of conditions of the form $p=\seq{M^p_\xi}{\xi\le \zeta^p}$, with $\zeta^p<\omega_1$, which are increasing continuous $\in$-chains through $X$, and ordered by end-extension.
\end{defn}

The forcing $\P_X$ plays a central role in this work, so let's for later reference state its main properties. For a proof, we refer to Fuchs \cite[Fact 2.9]{Fuchs:CanonicalFragmentsOfSRP}.

\begin{fact}
\label{fact:P_SisCountablyDistributiveAndExhaustive}
Let $\kappa$ be a regular uncountable cardinal, and let $X\sub[H_\kappa]^\omega$ be stationary in $[H_\kappa]^\omega$. Then
\begin{enumerate}[label=\textnormal{(\arabic*)}]
  \item
  \label{item:GenericSequenceHasLengthOmega1}
  for every countable ordinal $\zeta$, the set of conditions $p$ with $\zeta^p\ge\zeta$ is dense in $\P_X$,
  \item
  \label{item:UnionOfGenericIsExhaustive}
  for every $a\in H_\kappa$, the set of conditions $p$ such that there is a $\xi\le\zeta^p$ with $a\in M^p_\xi$ is dense in $\P_X$,
  \item
  \label{item:CountablyDistributive}
  $\P_X$ is countably distributive.
\end{enumerate}
\end{fact}

\begin{defn}
\label{def:FragmentsOfSRPandForcingAxioms}
For a class $\Gamma$ of forcing notions, the \emph{$\Gamma$-fragment of the strong reflection principle,} denoted $\Gamma$-$\SRP$, states that whenever $\kappa>\omega_1$ is a regular cardinal and $X\sub[H_\kappa]^\omega$ is stationary and $\P_X\in\Gamma$, then there is an increasing continuous $\in$-chain of length $\omega_1$ through $X$.

The strong reflection principle, denoted \SRP, is $\Gamma$-\SRP where $\Gamma$ is the class of all stationary set preserving forcing notions. And, fixing a regular cardinal $\kappa$, $\Gamma$-$\SRP(\kappa)$ is formulated like $\Gamma$-\SRP, but only for stationary sets $X\sub[H_\kappa]^\omega$.

We also write $\FA_{\omega_1}(\Gamma)$ for the forcing axiom for $\Gamma$ stating that given any forcing notion $\P\in\Gamma$ and any collection $\mathcal{D}$ of up to $\omega_1$ many dense subsets of $\P$, there is a filter in $\P$ meeting each set in $\mathcal{D}$. Martin's Maximum, introduced in \cite{FMS:MM1}, is the case where $\Gamma$ consists of all stationary set preserving forcing notions.

A strengthening of this forcing axiom, denoted $\FA^{++}(\Gamma)$, states that, given $\P\in\Gamma$, a collection $\mathcal{D}$ as before, and a set $\mathcal{T}$ of up to $\omega_1$ many $\P$-names such that for every $\dot{T}\in\mathcal{T}$, $\forces_\P$``$\dot{T}$ is a stationary subset of $\omega_1$,'' there is a filter $F$ in $\P$ which meets each member of $\mathcal{D}$ and interprets each $\dot{T}\in\mathcal{T}$ as a stationary subset of $\omega_1$, that is $\{\xi<\omega_1\st\exists p\in F\quad p\forces\check{\xi}\in\dot{T}\}$ is stationary.  
\end{defn}

If $X\sub[H_\kappa]^\omega$ is stationary, then for every $\zeta<\omega_1$, the set $D_\zeta$ consisting of all conditions in $\P_X$ of length greater than $\zeta$ is dense in $\P_X$, by Fact \ref{fact:P_SisCountablyDistributiveAndExhaustive}. In particular, the strong reflection principle follows from Martin's Maximum, and many of the major consequences of \MM follow already from \SRP. The ones that matter most to us here are:  
\begin{itemize}
\item the nonstationary ideal on $\omega_1$ is $\omega_2$-saturated (due to \Todorcevic, see \cite[Thm.~37.22]{ST3}),
\item exact $\Refl(\omega_1,S^\kappa_\omega)$ for regular $\kappa>\omega_1$ (essentially due to \Todorcevic, see \cite[Ex.~2.3]{FengJech:ProjectiveStationarityAndSRP} and \cite[Obs.~3.14]{Fuchs:CanonicalFragmentsOfSRP}),
\item  $2^\omega=\omega_2$ (by Fact \ref{fact:PointOfExactness}, these reflection principles imply that $2^\omega\le\omega_2^{\omega_1}=\omega_2$. The other direction follows as in the proof of Observation \ref{obs:Woodin} and is due to Woodin). 
\end{itemize}

On the other hand, a consequence of \MM that is not captured by \SRP is given by the following.

\begin{fact}[Larson {\cite[\S 4]{Larson:SeparatingSRP}}]
\label{fact:LarsonsSeparation}
$\MM$ implies $\OSR_{\omega_2}$, but assuming the consistency of $\MM$, $\SRP+\neg\OSR_{\omega_2}$ is consistent.
\end{fact}

The fragments of \SRP corresponding to forcing classes were introduced in Fuchs \cite{Fuchs:CanonicalFragmentsOfSRP}. Of particular interest for us is the class of $\infty$-subcomplete forcing notions. This class is a simplification due to Fuchs \& Switzer \cite{Fuchs-Switzer:IterationThmForSubversions} of the concept of subcompleteness, originally introduced by Jensen \cite{Jensen:SPSCF}, \cite{Jensen2014:SubcompleteAndLForcingSingapore}. In order to formulate this notion, we first need a technical concept due to Jensen.

\begin{defn}[Jensen]
\label{def:Full}
A transitive model $\calN$ of \ZFCm is \emph{full} if there is an ordinal $\gamma>0$ such that $L_\gamma(\calN)\models\ZFCm$ and $\calN$ is regular in $L_\gamma(\calN)$, meaning that if $a\in \calN$, $f:a\To \calN$ and $f\in L_\gamma(\calN)$, then $\ran(f)\in \calN$. A possibly nontransitive but well-founded model of \ZFCm is full if its transitive isomorph is full.
\end{defn}

We can now state when a forcing notion is $\infty$-subcomplete. We follow Jensen in using the notation $\calL_\tau^A=\kla{L_\tau[A],{\in},A\cap L_\tau[A]}$.

\begin{defn}
\label{def:(infinity-)subcompleteness}
A forcing notion $\P$ is $\infty$-\emph{subcomplete} if every sufficiently large cardinal $\theta$ \emph{verifies} the $\infty$-subcompleteness of $\P$, which means that $\P\in H_\theta$, and for any \ZFCm{} model $\calN=\calL_\tau^A$ with $\theta<\tau$ and $H_\theta\sub \calN$, any $\sigma:\bcalN\prec \calN$ such that $\bcalN$ is countable, transitive and full and such that $\P,\theta\in\ran(\sigma)$, any $\bar{G}\sub\bar{\P}$ which is $\bar{\P}$-generic over $\bcalN$, any $\bar{a}\in\bcalN$, the following holds. Letting $\sigma(\kla{\bar{\theta},\bar{\P}})=\kla{\theta,\P}$, there is a condition $p\in\P$ such that whenever $G\sub\P$ is $\P$-generic over $\V$ with $p\in G$, there is in $\V[G]$ a $\sigma':\bcalN\prec \calN$ such that
	\begin{enumerate}[label=(\arabic*)]
		\item
        \label{item:AgreementOnParameters}
        $\sigma'(\kla{\bar{a},\bar{\theta},\bar{\P}})=\sigma(\kla{\bar{a},\bar{\theta},\bar{\P}})$.
		\item
        \label{item:LiftingCondition}
        $(\sigma')``\bar{G}\sub G$.
	\end{enumerate}
\end{defn}

Jensen's original form of subcompleteness has an additional technical requirement, and is thus more restrictive than $\infty$-subcompleteness.
The crucial properties of $\infty$-subcomplete forcing are that they are iterable (using Miyamoto's nice iterations and do not add reals; see \cite{Fuchs-Switzer:IterationThmForSubversions}) and do not add reals. Yet they are able to change cofinalities and, most relevant to us here, kill suitable stationary subsets of regular cardinals greater than the continuum - we will give the details when we need them. 

\begin{defn}
\label{def:ForcingClasses}
We write $\SSP$ for the class of all forcing notions that preserve stationary subsets of $\omega_1$,  and $\infSC$ for the class of all $\infty$-subcomplete forcing notions. We write $\infty$-\SCFA for $\FA_{\omega_1}(\infty$-$\SCFA)$ and \SCFA for the forcing axiom for the original class of subcomplete forcing.
\end{defn}

The principle \SCFA, despite being consistent with \CH, has many of the consequences that \MM has; see Fuchs \cite{Fuchs:HierarchiesOfForcingAxioms}. An important consequence for us is that $\infty$-\SCFA implies $\DSR(\omega_1,S^\lambda_\omega)$, for every regular $\lambda>2^\omega$ (see Fuchs \cite{Fuchs:DiagonalReflection} and \cite{Fuchs:Errata}). Since the class of subcomplete forcing notions is contained in that of $\infty$-subcomplete forcing, $\SCFA$ is a consequence of $\infty$-\SCFA, and hence, all of these consequences of \SCFA also follow from $\infty$-\SCFA. And $\infty$-\SCFA remains compatible with \CH.

In the terminology of the previous definition, the original strong reflection principle \SRP is \SSP-\SRP. In order to apply this principle to a set $X\sub[H_\kappa]^\omega$, one must know whether $\P_X$ is in \SSP. This is where the following definition and fact come in.

\begin{defn}[Feng \& Jech \cite{FengJech:ProjectiveStationarityAndSRP}]
\label{def:ProjectiveStationarity}
Let $A$ be a set with $\omega_1\sub A$. Then,
a set $X\sub[A]^\omega$ is \emph{projective stationary in $[A]^\omega$} if for every stationary $T\sub\omega_1$, the set
\[X_T=\{x\in X\st x\cap\omega_1\in T\}\]
is stationary in $[A]^\omega$.
\end{defn}

\begin{fact}[Feng \& Jech \cite{FengJech:ProjectiveStationarityAndSRP}]
\label{fact:Feng-Jech}
A set $X\sub[H_\kappa]^\omega$ is projective stationary iff $\P_X$ is stationary set preserving. Thus, \SRP is equivalent to saying that for every regular cardinal $\kappa>\omega_1$ and for every set $X$ which is projective stationary in $[H_\kappa]^\omega$, there is an increasing continuous $\in$-chain of length $\omega_1$ through $X$.
\end{fact}

Since we want to work with the $\infty$-subcomplete fragment of \SRP, it will be useful to have a similar notion characterizing when a set $X\sub[H_\kappa]^\omega$ has the property that $\P_X$ is $\infty$-subcomplete. 

\begin{defn}[Fuchs \cite{Fuchs:CanonicalFragmentsOfSRP}]
\label{def:SpreadOut}
Let $B$ be a set (usually of the form $B=H_\kappa$, for some regular uncountable cardinal $\kappa$). A set $X\sub[B]^\omega$ with $\bigcup X=B$ is \emph{spread out} (in $B$) if for every sufficiently large cardinal $\theta$ with $X\in H_\theta$, whenever $\tau$, $A$, $M$ and $a$ are such that $H_\theta\sub \calL_\tau^A=\calN\models\ZFCm$, $\theta<\tau$, $X,a,\theta\in M \in [L_\tau [A]]^\omega$, $\calN|M\prec \calN$, and $\calN|M$ is full, then there are an $M' \in [L_\tau [A]]^\omega$ such that $\calN|M'\prec \calN$ and $M'\cap B\in X$, and an isomorphism $\pi:\calN|M\To \calN|M'$ such that $\pi(a)=a$.

If $\kappa$ is a cardinal and $X$ satisfies this property with the additional requirement that $\pi\rest\kappa$ is the identity, then $X$ is \emph{spread out above $\kappa$}.
\end{defn}

We remark that any mention of $a$ can be omitted in the preceding definition, because $A$ can be modified so as to make $a$ definable; see \cite[Def.~2.23]{Fuchs:CanonicalFragmentsOfSRP} and the following discussion. In verifying that a set is spread out, we may also assume that some useful parameter $b$ belongs to $M$ in the setting of the definition; the resulting concept was called \emph{weakly} spread out and shown to be equivalent to being spread out in \cite[Obs.~2.27]{Fuchs:CanonicalFragmentsOfSRP}. We will often make use of this fact.

\begin{fact}[Fuchs {\cite[Theorem 2.24]{Fuchs:CanonicalFragmentsOfSRP}}]
For an uncountable regular cardinal $\kappa$, a set $X\sub[H_\kappa]^\omega$ is spread out in $H_\kappa$ iff $\P_X$ is $\infty$-subcomplete. Thus, the principle $\infSC$-$\SRP (\kappa)$ holds iff every spread out subset of $[H_\kappa]^\omega$ contains an increasing continuous $\in$-chain of length $\omega_1$.
\end{fact}

A surprising takeaway from \cite{Fuchs:CanonicalFragmentsOfSRP} is that $\infty$-\SC-\SRP has many of the consequences that the full principle \SRP has, yet it is compatible with \CH. In fact, if one wants to have a local correspondence between these principles, \CH is in a sense necessary.

\begin{fact}[Fuchs {\cite[Obs.~2.32]{Fuchs:CanonicalFragmentsOfSRP}}]
\label{fact:CHneededForLocalCorrespondence}
The principle $\infty$-\SC-$\SRP(\kappa)$ holds for regular uncountable $\kappa\le 2^\omega$.
\end{fact}

Hence, in order for $\infty$-\SC-$\SRP(\omega_2)$ to have any nontrivial consequences (as $\SRP(\omega_2)$ has), \CH must be assumed. We will often do this. It should be noted that the principle $\infty$-\SCFA is consistent, assuming the existence of a supercompact cardinal, and in the model obtained by performing a nice iteration of $\infty$-subcomplete forcings, guided by a Laver diamond on the supercompact cardinal, $\infty$-\SCFA${}+ \CH$ holds (and hence so does $\infty$-\SC-\SRP); see \cite{Jensen:FAandCH}. This may be viewed as another argument that the assumption of \CH in addition to $\infty$-\SCFA or $\infty$-\SC-\SRP is a natural one. But of course, $\MM$ implies $\infty$-\SCFA${}+2^\omega=\omega_2$, so \CH does not follow from $\infty$-\SCFA. It is not known whether $\infty$-\SCFA is consistent with the continuum being larger than $\omega_2$.

In many ways, the spread out sets behave similarly to the projective stationary ones, in particular in the following sense.

\begin{obs}[Fuchs {\cite[Obs.~2.30]{Fuchs:CanonicalFragmentsOfSRP}}]
Let $\omega_1 \subseteq B_0 \sub A\sub B_1$, and let $X\sub [A]^\omega$ be spread out in $A$. Then $X\projectdown[B_0]^\omega$ and $X\projectup[B_1]^\omega$ are spread out in $B_0$ and $B_1$, respectively.
\end{obs}

The following lemma is our main tool in verifying that a set is spread out.

\begin{lem}[Fuchs {\cite[Lemma 3.5]{Fuchs:CanonicalFragmentsOfSRP}}]
\label{lem:OneStep}
Let $\kappa>2^\omega$ be a regular cardinal, $\theta>2^{2^\kappa}$ regular, $\calN$ a transitive model of $\ZFCm$ (in a finite language) with a definable well-ordering of its universe and $H_\theta\sub \calN$, $\sigma:\bcalN\prec \calN$, where $\bcalN$ is countable, transitive and full, with $\kappa\in\ran(\sigma)$. Let $\eta\in\kappa\cap\ran(\sigma)$ be such that $\eta^\omega<\kappa$. Let $\bar{\kappa},\bar{\eta}$ be the preimages of $\kappa,\eta$ under $\sigma$, respectively. Let $\bar{a}$ be some element of $\bcalN$.

Then there is an $\omega$-club of ordinals $\alpha<\kappa$ for which there is an embedding $\sigma':\bcalN\prec \calN$ with the following properties:
\begin{enumerate}[label=\textnormal{(\alph*)}]
  \item
  \label{item:MoveBasicParametersTheSameWay}
  Letting $p=\{\bar{a},\bar{\kappa},\bar{\eta}\}$, we have that $\sigma\rest p=\sigma'\rest p$.
  \item
  \label{item:SameBehaviorBelow}
  $\sigma\rest\bar{\eta}=\sigma'\rest\bar{\eta}$.
  \item
  \label{item:CorrectSupremum}
  $\sup\sigma'``\bkappa=\alpha$.
\end{enumerate}
\end{lem}


\section{Positive results}

The key concept for our positive results are the sets of the following kind.

\begin{defn}
\label{def:MagicalSets}
Let $\omega_1\le\kappa<\lambda$ be regular cardinals. Let $\vS$, $\vT$ be $\lambda$-sequences such that for $\alpha<\lambda$, $S_\alpha\sub S^\lambda_\omega$ is stationary in $\lambda$ and $T_\alpha\sub S^\kappa_\omega$ is stationary in $\kappa$, and such that for $\alpha<\beta<\lambda$, $T_\alpha\cap T_\beta$ is not stationary in $\kappa$.

Then $X_{\vS,\vT}$ is the set consisting of all $M\in [H_\lambda]^\omega$ such that either there is no $\alpha\in M\cap\lambda$ with $\sup(M\cap\kappa)\in T_\alpha$, in which case we require that $\sup(M\cap\lambda)\in S_0$, or there is a unique such $\alpha$, and then $\sup(M\cap\lambda)\in S_\alpha$. 
\end{defn}

Most of our results rely on the following lemma.

\begin{lem}
\label{lem:MagicalSetsSpreadOutAboveKappa}
Let $\omega_1\le\kappa<\lambda$ be regular cardinals, and suppose that $\kappa^\omega<\lambda$. Let $\vS$, $\vT$ be $\lambda$-sequences as in Definition \ref{def:MagicalSets}. Then $X=X_{\vS,\vT}$ is spread out above $\kappa$.
\end{lem}

\begin{proof}
In the setting of Definition \ref{def:SpreadOut}, let $\calN|M\prec \calN=\calL_\tau^A\supseteq H_\theta$, with $\vS, \vT, X\in M$, and fix some $a\in M$. Let $\gamma=\sup(M\cap\kappa)$.

\medskip
\noindent\emph{Case 1:} There is an $\alpha\in M\cap\lambda$ such that $\gamma\in T_\alpha$. 

This $\alpha$ is then unique, because if $\beta\in M\cap\lambda$ is different from $\alpha$, then there is a $C\in M$ which is a club subset of $\kappa$ such that $T_\alpha\cap T_\beta\cap C=\leer$ (we use here that $\vT\in M$, so $T_\alpha$ and $T_\beta$ are also in $M$). Clearly, $\gamma\in C$, so it cannot be that $\gamma\in T_\beta$. 

By Lemma \ref{lem:OneStep}, since $\kappa^\omega<\lambda$, there is an $M'$ such that there is an isomorphism $\pi:\calN|M\To \calN|M'$ that fixes $\alpha,\vS,\vT,X$ and $a$, so that $M\cap\kappa=M'\cap\kappa$ (this is the extra condition for being spread out above $\kappa$), and such that $\sup(M'\cap\lambda)\in S_\alpha$. Then $\gamma=\sup(M'\cap\kappa)$, and by the same argument as above, $\alpha$ is the unique $\beta\in M'\cap\lambda$ such that $\gamma\in T_\beta$. Hence, $M'\cap H_\lambda\in X$.

\medskip
\noindent\emph{Case 2:} There is no $\alpha\in M\cap\lambda$ such that $\gamma\in T_\alpha$. 

Then, again by Lemma \ref{lem:OneStep}, there is an $M'\prec \calN$ such that $\calN|M\isomorphic \calN|M'$, the isomorphism fixes any desired parameter, $M\cap\kappa=M'\cap\kappa$, and $\sup(M'\cap\lambda)\in S_0$. It follows that $M'\cap H_\lambda\in X$, unless there is an $\alpha\in M'\cap\lambda$ such that $\gamma\in T_\alpha$ (and $\sup(M'\cap\lambda)\notin S_\alpha$). If so, we can run the argument from case 1 to find an $M''$ isomorphic to $M'$ with $M''\cap\kappa=M'\cap\kappa$ and $M''\cap H_\lambda\in X$. 

\medskip
Thus, in both cases, we found an isomorphic copy of $M$ whose intersection with $H_\lambda$ is in $X$, such that the isomorphism is the identity below $\kappa$, verifying that $X$ is spread out above $\kappa$.
\end{proof}

We also use the following lemma.

\begin{lem}
\label{lem:StronglyExactSimultaneousRP}
Let $\lambda > \omega_1$ be a regular cardinal and $\vS$ be a $\lambda$-sequence such that for $\alpha < \lambda$, $S_\alpha \subseteq S^\lambda_\omega$ is stationary in $\lambda$. Suppose $\seq{M_\xi}{\xi < \omega_1}$ is an increasing continuous $\in$-chain of countable sets with $\kla{M_\xi,{\in}} \prec \kla{H_\lambda,{\in}}$, and let $R = \bigcup_{\xi<\omega_1} M_\xi \cap \lambda$ and $\rho = \sup R$. Suppose there exists a sequence $\seq{T_\alpha}{\alpha \in R}$ of stationary subsets of $\omega_1$ with the following properties.
\begin{enumerate}[label=\textnormal{(\arabic*)}]
\item For any $\xi < \omega_1$ and any $\alpha \in R$, if $\alpha \in M_\xi \cap \lambda$ and $M_\xi \cap \omega_1 \in T_\alpha$, then $\sup ( M_\xi \cap \lambda ) \in S_\alpha$.
\item For any $\xi < \omega_1$ there is $\alpha \in M_\xi \cap \lambda$ with $M_\xi \cap \omega_1 \in T_\alpha$.
\end{enumerate}
Then $R$ contains a club in $\rho$, and $\rho$ is a strongly exact simultaneous reflection point of $\vS | R$.
\end{lem}

\begin{proof}
Let $\rho_\xi = \sup ( M_\xi \cap \lambda )$. Then $\{ \rho_\xi \st \xi < \omega_1 \}$ is club in $\rho$ and is included in $R$. Below we prove the latter statement.

Let $C = \{ \xi < \omega_1 \st M_\xi \cap \omega_1 = \xi \}$. Then $C$ is club in $\omega_1$. For each $\alpha \in R$, by shrinking $T_\alpha$ if necessary, we may assume that $T_\alpha \subseteq C$ and $\alpha \in M_\xi \cap \lambda$ for all $\xi \in T_\alpha$. Then by (1) if $\xi \in T_\alpha$, then $\rho_\xi \in S_\alpha$. Then $S_\alpha \cap \rho$ is stationary for each $\alpha \in R$ since $T_\alpha$ is stationary in $\omega_1$, and the map sending $\xi$ to $\rho_\xi$ is increasing continuous and cofinal in $\rho$. Thus $\rho$ is a simultaneous reflection point of $\vS |R$.

To see that $\rho$ is a strongly exact reflection point of $\vS |R$, take an arbitrary stationary $U \subseteq \rho$. We find $\alpha \in R$ with $S_\alpha \cap U$ stationary. Since $U$ is stationary, we have that $X = \{ M_\xi \mid \xi < \omega_1 , \ \rho_\xi \in U \}$ is stationary in $[ M_{\omega_1} ]^\omega$, where $M_{\omega_1} = \bigcup_{\xi < \omega_1} M_\xi$. By (2), for each $M \in X$ take $\alpha_M \in M \cap \lambda$ with $M \cap \omega_1 \in T_{\alpha_M}$. By the version of Fodor's theorem for stationary subsets of $[ M_{\omega_1} ]^\omega$ (\cite[Thm.~8.24]{ST3}), we can take $\alpha \in R$ and a stationary $Y \subseteq X$ such that $\alpha_M = \alpha$ for all $M \in Y$. Now let $W = \{ \sup ( M \cap \lambda ) \st M \in Y \}$. Then $W$ is stationary in $\rho$, $W \subseteq U$, and $W \subseteq S_\alpha$ by (1). Thus $S_\alpha \cap U$ is stationary.
\end{proof}

\begin{thm}
\label{thm:GettingeDSR}
Let $\lambda>2^\omega$ be a regular cardinal such that
$\infSC$-$\SRP(\lambda)$ holds and $\NS_{\omega_1}$ is not $\lambda$-saturated. Then strongly exact $\DSR(\omega_1,S^\lambda_\omega)$ holds.
\end{thm}

\begin{proof}
Let $\vS = \seq{S_{\alpha,\xi}}{\alpha<\lambda,\ \xi<\omega_1}$ be a sequence of stationary subsets of $S^\lambda_\omega$. We find a $\rho < \lambda$ and an $R \subseteq \rho$ such that $R$ contains a club in $\rho$, and $\rho$ is a strongly exact simultaneous reflection point of $\vS |R$.

Since $\NS_{\omega_1}$ is not $\lambda$-saturated, we may fix a sequence $\seq{T_\alpha}{\alpha<\lambda}$ of pairwise almost disjoint stationary subsets of $\omega_1$, in the sense that if $\alpha<\beta<\lambda$, then $T_\alpha\cap T_\beta$ is nonstationary.
Fix a bijection $f:\lambda\To\lambda\times\omega_1$. Let $S_\alpha ' =S_{f(\alpha)}$ for $\alpha < \lambda$, and
let $\vS ' =\seq{S_\alpha '}{\alpha < \lambda}$.

Let $X=X_{\vS ',\vT}$ be the set from Definition \ref{def:MagicalSets}. By Lemma \ref{lem:MagicalSetsSpreadOutAboveKappa}, $X$ is spread out (above $\omega_1$, but every spread out set is). Let $Z$ be the set of all $M \in [H_\lambda]^\omega$ with $\kla{M,{\in},f\rest M}\prec\kla{H_\lambda,{\in},f}$.
Then $Z$ is club, and so, $X\cap Z$ is also spread out, by \cite[Obs.~2.29]{Fuchs:CanonicalFragmentsOfSRP}. By $\infSC$-$\SRP(\lambda)$, let $\seq{M_\xi}{\xi< \omega_1}$ be an increasing continuous $\in$-chanin through $X\cap Z$. Let $M_{\omega_1} = \bigcup_{\xi < \omega_1} M_\xi$, $R = M_{\omega_1} \cap \lambda$ and $\rho = \sup R$.

Note that the set
\[ T_0 ' = T_0 \cup \{ \xi < \omega_1 \st \text{there is no $\alpha \in M_\xi \cap \lambda$ with $M_\xi \cap \omega_1 \in T_\alpha$} \} \]
together with $\seq{T_\alpha}{\alpha \in R \setminus \{ 0 \}}$ satisfies the assumption of Lemma \ref{lem:StronglyExactSimultaneousRP}.
Thus $R$ contains a club in $\rho$, and $\rho$ is a strongly exact simultaneous reflection point of $\vS ' |R$.
But $\vS '|R = \vS |R$ since $M_{\omega_1}$ is closed under $f$ and $f^{-1}$. Thus $\rho$ and $R$ are as desired.
%
%
\end{proof}

\subsection{When $\CH$ holds}
\label{subsec:WhenCHholds}

Here, we consider $\infSC$-$\SRP$ together with \CH. It was  argued in Fuchs \cite{Fuchs:CanonicalFragmentsOfSRP} that the correct principle to consider is $\infSC$-$\SRP+\CH$; see also Fact \ref{fact:CHneededForLocalCorrespondence} and the following remark. We begin by noting that the addition of \CH can be considered to be a slight strengthening of the 
 failure of saturation of the nonstationary ideal.

\begin{obs}
\label{obs:Woodin}
$\infSC$-$\SRP$ $+$ $\CH$ implies that the nonstationary ideal on $\omega_1$ is not $\omega_2$-saturated.
\end{obs}

\begin{proof}
Assume $\infSC$-$\SRP + \CH$. 

By \cite[Thm.~3.19]{Fuchs:CanonicalFragmentsOfSRP}, a principle called the strong Friedman property holds at every regular $\kappa>2^\omega=\omega_1$. By \cite[Obs.~3.14]{Fuchs:CanonicalFragmentsOfSRP}, this property implies exact $\Refl(\omega_1,S^\kappa_\omega)$ for every regular $\kappa>2^\omega$. And by a folklore result, already $\Refl(2,S^\kappa_\omega)$ implies the failure of the principle $\Box(\kappa)$; for a proof we refer to \cite[Thm.~2.1]{Hayut-LambieHanson:SimultaneousReflectionAndSquare}.

Thus, under our assumption, $\square(\kappa)$ fails, for every regular cardinal $\kappa\ge\omega_2$. In particular, for every regular cardinal $\kappa\geq\omega_2$, both $\square(\kappa)$ and $\square_\kappa$ fail (as $\square_\kappa$ implies $\square(\kappa^+)$). This implies, by \cite[Proposition 2.2]{Schimmerling:CoherentSequencesAndThreads}, that $A^\#$ exists, for every $A\sub\kappa^+$. So, by choosing $\kappa\ge 2^{\omega_1}$, $\power(\omega_1)^\#$ exists. 

If the nonstationary ideal on $\omega_1$ were saturated, then this would imply by Woodin \cite[Theorem 3.17]{Woodin:ADforcingAxiomsNonstationaryIdeal} that $\undertilde{\delta}{}^1_2=\omega_2$. But this contradicts \CH, since if $\undertilde{\delta}{}^1_2=\omega_2$, then there are at least $\omega_2$ $\undertilde{\Sigma}{}^1_2$-definable prewellorderings of $\reals$, so since there are only continuum many (boldface) projective sets of reals, $\omega_2\le 2^\omega$.
\end{proof}

\begin{cor}
Assume $\infSC$-$\SRP(\omega_2)$ $+$ \CH $+$ $\NS_{\omega_1}$ is not $\omega_2$-saturated. Then strongly exact $\DSR(\omega_1,S^{\omega_2}_\omega)$ holds, and hence, 
$\OSR_{\omega_2}$ holds.
\end{cor}

\begin{proof}
By Theorem \ref{thm:GettingeDSR}, strongly exact $\DSR(\omega_1,S^{\omega_2}_\omega)$ holds, and by \cite[Lemma 2.4]{Fuchs-LambieHanson:SeparatingDSR}, this implies $\OSR_{\omega_2}$.
\end{proof}

In fact, strongly exact $\DSR(\omega_1,S^{\omega_2}_\omega)$ implies strongly exact $\OSR_{\omega_2}$. This can be seen by applying strongly exact $\DSR(\omega_1,S^{\omega_2}_\omega)$, in the equivalent form of strongly exact $\DSR({<}\omega_2,S^{\omega_2}_\omega)$, to the sequence $\seq{S_{\alpha,i}}{i<\alpha<\omega_2}$ defined by $S_{\alpha,i}=S_i$.  

This immediately gives the main result of this section.

\begin{thm}
\label{thm:SC-SRP+CHimpliesOSR}
$\infSC$-\SRP $+$ $\CH$ implies strongly exact $\DSR(\omega_1,S^{\omega_2}_\omega)$ and hence strongly exact $\OSR_{\omega_2}$.
\end{thm}

\begin{proof}
By Observation \ref{obs:Woodin}, $\NS_{\omega_1}$ is not $\omega_2$-saturated, so the previous corollary applies.
\end{proof}

This should be contrasted with Fact \ref{fact:LarsonsSeparation}, stating that $\SRP$ does not imply $\OSR_{\omega_2}$. Thus, the principle $\infty$-\SC-\SRP, though weaker than \SRP, has stronger consequences in terms of diagonal reflection, when it is boosted with (a mild strengthening of) nonsaturation. 

\begin{remark}
Under $\infSC$-\SRP $+$ \CH, exact $\refl(\omega_1,S^{\omega_2}_\omega)$ holds, which implies, by Fact \ref{fact:PointOfExactness}, that $\omega_2^{\omega_1}=\omega_2$. Thus, $\NS_{\omega_1}$ is $\omega_3$-saturated.
\end{remark}

Thus, Theorem \ref{thm:GettingeDSR} is only applicable to $\lambda=\omega_2$ under $\infSC$-\SRP $+$ \CH. If we wanted to apply the theorem to a regular $\lambda>\omega_2$, we would have to have that $\NS_{\omega_1}$ is not $\lambda$-saturated, so in particular, $2^{\omega_1}\ge\lambda$. It then would have to be the case that $\theta:=(2^\omega)^+\ge\lambda$, because $\infSC$-\SRP implies exact $\refl(\omega_1,S^\theta_\omega)$, which implies that $\theta^{\omega_1}=\theta$, so if we had that $\theta<\lambda$, then we would have $2^{\omega_1}\le\theta<\lambda$, a contradiction. 
In order to apply the theorem, we would also need that $2^\omega<\lambda$. Thus, it would have to be that $\lambda=(2^\omega)^+\ge\omega_3$, and it is natural to ask whether this scenario is possible.

\begin{question}
Is the theory \[\ZFC + \infSC\text{-}\SRP((2^\omega)^+)+\neg\CH + \text{``$\NS_{\omega_1}$ is not $(2^\omega)^+$-saturated''}\] consistent?
\end{question}

In the next two subsections, we will find ways to apply the argument of the proof of Theorem \ref{thm:GettingeDSR} to $\lambda>\omega_2$ by  combinations of strengthening the strong reflection principle and the nonsaturation assumption.

\subsection{Using $\Gamma$-$\SRP^{++}$}

For a forcing class $\Gamma$, the $\Gamma$-fragment of $\SRP$ follows from the forcing axiom for the class $\Gamma^-$ introduced in the next definition.

\begin{defn}
\label{def:GammaPrime}
For a forcing class $\Gamma$, let
\[\Gamma^-=\Gamma\cap\{\P_X\st\text{there is a regular}\ \lambda>\omega_1\ \text{such that}\ X\sub[H_\lambda]^\omega\ \text{is stationary}\}.\]
\end{defn}

It is unknown whether $\FA_{\omega_1}(\Gamma^-)$ is actually equivalent to $\Gamma$-\SRP when $\Gamma$ is one of the classes of interest to us here. We view it as a ``forcing axiom version'' of the strong reflection principle, hence the following terminology.

\begin{defn}
\label{def:FAversionOfSRP}
Let $\Gamma$ be a forcing class. Then $\Gamma$-$\SRP^\FA$ denotes $\FA_{\omega_1}(\Gamma^-)$. In the case where $\Gamma=\SSP$, we abbreviate this by $\SRP^\FA$.
\end{defn}

Thus, $\Gamma$-$\SRP^\FA$ implies $\Gamma$-$\SRP$, but the converse is not known. We will return to $\SRP^\FA$ in Section \ref{sec:Separations}, but for now, we need a strengthening.

\begin{defn}
Let $\Gamma$ be a forcing class. Then $\Gamma$-$\SRP^{++}$ stands for the forcing axiom  $\FA_{\omega_1}^{++}(\Gamma^-)$. As expected, $\SRP^{++}$ denotes $\SSP$-$\SRP^{++}$. 
\end{defn}

\begin{lem}
\label{lem:SRP++ImpliesExactDSR}
$\infSC$-$\SRP^{++}+\CH$ implies strongly exact $\DSR(\omega_1,S^\lambda_\omega)$, for every regular cardinal $\lambda > \omega_1$. 
\end{lem}

\noindent\emph{Note:} 

\begin{itemize}
\item For $\lambda=\omega_2$, it follows already from $\infSC$-$\SRP+\CH$.
\item For $\lambda>\omega_2$, we do not need $\CH$; it is sufficient to know that $\lambda > 2^\omega$. This holds under $\SRP$ for every regular cardinal greater than $\omega_2$, so we get as a corollary that $\SRP^{++}$ implies strongly exact $\DSR(\omega_1,S^\lambda_\omega)$ for every regular $\lambda>\omega_2$.
\item $\DSR(\omega_1,S^\lambda_\omega)$ follows already from $\MA^{++}(\sigma\text{-closed})$.
\end{itemize}

\begin{proof}
Let a sequence $\vS=\seq{S_{\alpha,\xi}}{\alpha < \lambda,\ \xi < \omega_1}$ of stationary subsets of $S^\lambda_\omega$ be given. We find $\rho < \lambda$ and $R \subseteq \rho$ witnessing strongly exact $\DSR ( \omega_1 , S^\lambda_\omega )$ for $\vS$.

Fix a bijection $f:\lambda\To\lambda\times\omega_1$. Let $S_\alpha ' =S_{f(\alpha)}$ for $\alpha < \lambda$, and
let $\vS ' =\seq{S_\alpha '}{\alpha < \lambda}$. Let $X$ be the set of all $M\in [H_\lambda]^\omega$ such that $\kla{M,{\in},f\rest M}\prec\kla{H_\lambda,{\in},f}$ and there is an $\alpha \in M \cap \lambda$ with $\sup ( M \cap \lambda ) \in S_\alpha '$. $X$ is spread out by Lemma \ref{lem:OneStep} and \cite[Obs.~2.29]{Fuchs:CanonicalFragmentsOfSRP}. Let $\P=\P_X$. If $G$ is $\P$-generic, then $\bigcup G$ is a sequence $\seq{M_\xi}{\xi<\omega_1}$. Let $\dot{M}_\xi$ be a name for the $\xi$-th element of $\bigcup\dot{G}$, where $\dot{G}$ is the canonical $\P$-name for the generic filter. Further, for $\alpha<\lambda$, let $\dot{T}_\alpha$ be a $\P$-name such that
\[\forces_{\P}\dot{T}_\alpha =\{ \dot{M}_\xi \cap \omega_1 \st \check{\xi} < \omega_1\ \land\ \check{\alpha} \in\dot{M}_\xi\ \land\ \sup(\dot{M}_\xi\cap\check{\lambda})\in\check{S}_\alpha '\}.\]
\begin{NumberedClaim}
\label{NC:BxiStationary}
For any $\alpha<\lambda$, $\forces_\P\text{``$\dot{T}_\alpha$ is stationary in $\omega_1$.''}$	
\end{NumberedClaim}

\begin{proof}[Proof of \ref{NC:BxiStationary}.]
Fix $\alpha < \lambda$. Let $\dot{C}$ be a $\P$-name and $p\in\P$ a condition that forces that $\dot{C}$ is a club in $\omega_1$. We find $q\leq p$ and $\zeta < \omega_1$ such that $q\forces_\P\text{``$\check{\zeta}\in \dot{T}_\alpha \cap \dot{C}$.''}$

Let $\theta$ be a sufficiently large regular cardinal and $\calN = \kla{H_\theta,\in}$.
Take $M \in [H_\theta]^\omega$ with $\dot{C},X,\vS ' ,p,\alpha\in M$, $\calN|M\prec\kla{H_\theta,{\in}}$ and $\sup (M\cap\lambda) \in S_\alpha '$, and let $\zeta = M \cap \omega_1$. Note that $M\cap H_\lambda \in X$. Take an enumeration $\seq{D_n}{n<\omega}$ of all dense subsets of $\P$ in $M$ and a descending sequence $\seq{p_n}{n<\omega}$ in $\P\cap M$ below $p$ such that $p_n \in D_n$. Let $q' = \bigcup_{n<\omega} p_n$. Since for every $\xi<\omega_1$, the set of conditions in $\P$ of length at least $\xi$ is dense, the length of $q'$ is $\zeta$. Since for every $a\in H_\lambda$, the set of conditions $r$ such that $a\in\bigcup\ran(r)$ is dense, $\bigcup\ran(q')=M \cap H_\lambda \in X$. Hence, $q=q'\verl\kla{M\cap H_\lambda}$ is a condition in $\P$ that extends $p$.

Since $\alpha\in M$, $\sup (M\cap\lambda)\in S_\alpha '$, and $q\forces_{\P} \text{``$\dot{M}_\zeta = M \cap H_\lambda$''}$, we have that $q$ forces $\zeta = M \cap \omega_1 \in \dot{T}_\alpha$. Since for every $\xi<\omega_1$, the set of conditions $r$ such that there is $\eta\ge\xi$ with $r\forces_{\P}\text{``$\check{\eta}\in\dot{C}$''}$ is dense, it follows that $q$ forces that $\check{\zeta}$ is a limit point of $\dot{C}$, and hence that $\check{\zeta}\in\dot{C}$. Thus $q$ and $\zeta$ are as desired. \hspace*\fill$\qedhere_{\ref{NC:BxiStationary}}$\end{proof}

Let $\theta$ be a sufficiently large regular cardinal and $\calN=\kla{H_\theta,{\in}}$.
By \infSC-$\SRP^{++}$, using a folklore characterization as given by Jensen \cite[\S1, Lemma 2]{Jensen:FAandCH}, we can now take $\sigma:\bcalN\prec \calN$ with $\bcalN$ transitive and $\omega_1\cup\{\vS,X,\P\}\sub\ran(\sigma)$, and, denoting preimages under $\sigma$ by adding bars, can take a $\bP$-generic filter $\bG$ over $\bcalN$ such that for every $\alpha<\blambda$, $\bT_\alpha=\dot{\bT}_\alpha^{\bG}$ is stationary in $\omega_1$, where $\seq{\dot{\bT}_\alpha}{\alpha < \blambda} = \sigma^{-1} ( \seq{\dot{T}_\alpha}{\alpha < \lambda} )$. Let $R = ( \sigma``\bcalN ) \cap \lambda$ and $\rho = \sup R$. We show that $R$ contains a club in $\rho$, and $\rho$ is a strongly exact simultaneous reflection point of $\vS | R$.

Let $\seq{\bM_\xi}{\xi<\omega_1}=\bigcup\bG$. So $\seq{\bM_\xi}{\xi<\omega_1}$ is an increasing continuous $\in$-chain through $\bX$, and $\bigcup_{\xi<\omega_1} \bM_\xi = H_\blambda^{\bcalN}$. Note that, by the definition of $\dot{T}_\alpha$, for any $\xi < \omega_1$ and any $\alpha < \blambda$, if $\bM_\xi \cap \omega_1 \in \bT_\alpha$, then $\sup ( \bM_\xi \cap \blambda ) \in \bS_\alpha '$, where $\seq{\bS_\alpha '}{\alpha<\blambda} = \sigma^{-1}(\seq{S_\alpha '}{\alpha < \lambda})$. Note also that for any $\xi < \omega_1$ there is $\alpha \in \bM_\xi \cap \blambda$ with $\bM_\xi \cap \omega_1 \in \bT_\alpha$ since $\bM_\xi \in \bX$.

Let $M_\xi = \sigma ( \bM_\xi )$ for $\xi<\omega_1$. Note that $\seq{M_\xi}{\xi < \omega_1}$ is an increasing continuous $\in$-chain through $X$, and $\bigcup_{\xi < \omega_1} M_\xi \cap \lambda = R$. For $\alpha \in R$ let $T_\alpha = \sigma ( \bT_{\balpha} ) = \bT_{\balpha}$. Then $\seq{T_\alpha}{\alpha\in R}$ satisfies the assumption of Lemma \ref{lem:StronglyExactSimultaneousRP} for $\vS '$. So $R$ contains a club in $\rho$, and $\rho$ is a strongly exact simultaneous reflection point of $\vS ' |R$. Moreover $\vS ' |R = \vS |R$ since $\bigcup_{\xi < \omega_1} M_\xi$ is closed under $f$ and $f^{-1}$. Thus $\rho$ and $R$ are as desired.
\end{proof}

Of course, assuming $\infty$-\SC-$\SRP^{++}$ is a big step. The result is different from the others in that in the case $\lambda>\omega_2$, which is where it goes beyond Theorem \ref{thm:GettingeDSR}, it is compatible with $\NS_{\omega_1}$ being saturated, and in fact works under the assumption of $\SRP^{++}$ as well. Here, assuming $\SRP^{\FA}$ is not enough, as we shall see in Subsection \ref{subsec:SeparatingEnhancedSRP-FAfromuDSR}, where we will show that an even stronger assumption than $\SRP^\FA$ does not imply an even weaker diagonal reflection principle. However, in the context of \CH, we do not know whether $\infty$-\SC-$\SRP^{\FA}$ suffices, whereas in the next subsection, we will exhibit a slightly stronger principle that does.

\subsection{When $\diamondsuit$ holds}
\label{subsec:diamond}

In subsection \ref{subsec:WhenCHholds}, we worked under two assumptions: $\infty$-\SC-\SRP and \CH. Here, we will slightly strengthen both of these. Instead of \CH, we will assume $\diamondsuit$, and we will consider an enhanced version of $\infty$-\SC-$\SRP^\FA$. Of course, $\diamondsuit$ implies the existence of $2^{\omega_1}$ almost disjoint (with respect to the ideal of bounded sets) stationary subsets of $\omega_1$, and hence can be viewed as a strong failure of the saturation of $\NS_{\omega_1}$. Thus, $\diamondsuit$ is an assumption that fits perfectly into the context of this work. But aside from this, since $\diamondsuit$ can be forced to hold by a countably closed, and hence $\infty$-subcomplete forcing, and since $\infty$-subcomplete forcing preserves $\diamondsuit$, the natural model of $\infty$-\SCFA obtained by forcing over a model with a supercompact cardinal satisfies $\diamondsuit$. So it is not only a fitting, but also a consistent and natural assumption to make.

\begin{defn}
Let $\Gamma$ be a forcing class. Let $\Gamma^*$ be the class of all forcing notions in $\Gamma$ that are of the form $\P*\dot{\Q}$, where $\P$ is trivial forcing or of the form $\Add(\omega_1,\lambda)$, for some cardinal $\lambda$, and $\dot{\Q}$ is a $\P$-name for a forcing such that
\[\forces_\P\text{``there are a regular $\kappa>\omega_1$ and a stationary $X\sub[H_{\kappa}]^\omega$ such that $\dot{\Q}=\P_X$.''}\]
Here, $\Add(\omega_1,\lambda)$ is the usual forcing to add $\lambda$ many subsets of $\omega_1$ with countable conditions.
We refer to $\FA_{\omega_1}(\Gamma^*)$ as the \emph{enhanced $\Gamma$-$\SRP^\FA$.} As before, if $\Gamma=\SSP$, we just write \emph{enhanced $\SRP^\FA$}.
\end{defn}

We have the implications $\FA_{\omega_1}(\Gamma)$ $\implies$ enhanced $\Gamma$-$\SRP^\FA$ $\implies$ $\Gamma$-$\SRP^\FA$.

Note also that under \CH, $\Add(\omega_1,\lambda)$ is $\omega_2$-c.c., and it is countably closed, so it preserves cardinals and cofinalities. Thus, under \CH, the step from $\infSC$-$\SRP^\FA$ to enhanced $\infSC$-$\SRP^\FA$ may seem like a small one, but the next theorem shows that if one simultaneously moves from \CH to $\diamondsuit$, one can run the argument of Theorem \ref{thm:GettingeDSR} to great effect.

\begin{thm}
\label{thm:AddingDiamondGivesDiagonalReflection}
Let $\lambda>\omega_1$ be a regular cardinal.
Then
\[(\text{enhanced $\infSC$-$\SRP^\FA+\diamondsuit)\implies$ strongly exact $\DSR(\omega_1,S^\lambda_\omega)$.}\]
\end{thm}

\noindent\emph{Note:} The result is local, and slightly stronger than stated. Namely, the forcing axiom for forcing notions of the form $\Add(\omega_1,\lambda)*\dot{\Q}$, where $\Add(\omega_1,\lambda)$ forces that there is a spread out subset of $[H_\lambda]^\omega$ such that $\dot{\Q}=\P_X$ (i.e., $\P_X\in\infty$-$\SC$), is sufficient. To emphasize, asking that $\Add(\omega_1,\lambda)*\dot{\Q}$ is in $\infSC$ is weaker than asking that $\Add(\omega_1,\lambda)$ forces that $\dot{\Q}$ is in $\infSC$, hence this modified forcing axiom follows from the enhanced \infSC-$\SRP^\FA$.

\begin{proof}
Let a sequence $\vS = \seq{S_{\alpha,\xi}}{\alpha<\lambda,\ \xi<\omega_1}$ of stationary subsets of $S^\lambda_\omega$ be given. We find $\rho < \lambda$ and $R \subseteq \rho$ witnessing strongly exact $\DSR ( \omega_1 , S^\lambda_\omega )$ for $\vS$.

Fixing a bijection $f:\lambda\To\lambda\times\omega_1$, define $S_\alpha ' =S_{f(\alpha)}$, for $\alpha<\lambda$, as in the proof of Theorem \ref{thm:GettingeDSR}. Denote $\seq{S_\alpha '}{\alpha<\lambda}$ by $\vS '$.
Fix a $\diamondsuit$-sequence $\vD=\seq{D_\xi}{\xi<\omega_1}$.

Let $\P=\Add(\omega_1,\lambda)$. Then $\P$ is $\omega_2$-c.c., so it preserves $\lambda$ as a regular cardinal, and it is $\omega_1$-closed, so it preserves $S^\lambda_\omega$, stationary subsets of $S^\lambda_\omega$, and $\vec{D}$ as a $\diamondsuit$-sequence. For $\alpha<\lambda$, let $\dot{A}_\alpha$ be the canonical $\P$-name for the $\alpha$-th new subset of $\omega_1$ added by $\P$, and let $\dot{T}_\alpha$ be a canonical $\P$-name for the set
\[\{\xi<\omega_1\st\dot{A}_\alpha\cap\xi=D_\xi\}.\]
Thus, since $\vec{D}$ remains a $\diamondsuit$-sequence in $\V^\P$, it is forced by $\P$ that the $\dot{T}_\alpha$ are a pairwise almost disjoint sequence of stationary subsets of $\omega_1$.

Let $\dot{X}$ be the canonical $\P$-name for the set $X_{\vS ',\dot{\vT}}$ of Definition \ref{def:MagicalSets}. By Lemma \ref{lem:MagicalSetsSpreadOutAboveKappa}, $\dot{X}$ is spread out in $\V^\P$. Also, let $\dot{Z}$ be the canonical $\P$-name for the set of all $M \in [H_\lambda]^\omega$ with $\kla{M,{\in},f\rest M}\prec\kla{H_\lambda^{\V^\P},{\in},f}$.
As before, letting $\dot{X}'$ be the canonical $\P$-name for $\dot{X} \cap \dot{Z}$, we have that $\dot{X}'$ is spread out in $\V^\P$. Let $\dot{\Q}$ be the canonical $\P$-name for the forcing $\P_{\dot{X}'}$.

Now take a sufficiently large regular cardinal $\theta$, and let $\calN = \kla{H_\theta,{\in}}$. By our assumption, $\FA_{\omega_1}(\{\P*\dot{\Q}\})$ holds, so, appealing to Jensen \cite[\S1, Lemma 1]{Jensen:FAandCH}, for example, let $\bcalN$ be transitive, and let $\sigma:\bcalN\prec \calN$ be such that $\omega_1\cup\{f,\vec{S},\vec{D}\}\sub\ran(\sigma)$, and such that, denoting preimages under $\sigma$ by adding bars as before, there is a filter $\bar{G}*\bar{H}$ which is $\bar{\P}*\dot{\bar{\Q}}$-generic over $\bcalN$. Let $R = ( \sigma``\bcalN ) \cap \lambda$ and $\rho = \sup R$. We show that $\rho$ and $R$ are as desired.

Since $\omega_1\sub\ran(\sigma)$ and $\sigma\restrict\omega_1=\id$, we have $\sigma^{-1}(\vec{D})=\vec{D}$ is our original $\diamondsuit$-sequence. Denote $\sigma^{-1}(\seq{\dot{A}_\alpha}{\alpha<\lambda})$ by $\seq{\dot{\bA}_\alpha}{\alpha<\blambda}$,
$\sigma^{-1}(\seq{\dot{T}_\alpha}{\alpha<\lambda})$ by $\seq{\dot{\bar{T}}_\alpha}{\alpha<\blambda}$
and let $\bar{A}_\alpha=(\dot{\bar{A}}_\alpha)^{\bG}$, as well as $\bar{T}_\alpha=(\dot{\bar{T}}_\alpha)^{\bG}$, for $\alpha<\blambda$. Then, by the choice of these names, $\bar{T}_\alpha=\{\xi<\omega_1\st \bA_\alpha\cap\xi=D_\xi\}$ is a stationary subset of $\omega_1$. (This is because $\vec{D}$ is a $\diamondsuit$-sequence not only in $\bcalN$, but in $\V$ - it is to this end that we assumed $\diamondsuit$ to hold.)

Let $\seq{\bM_\xi}{\xi<\omega_1}=\bigcup\bH$. So $\seq{\bM_\xi}{\xi<\omega_1}$ is an increasing continuous $\in$-sequence through $\bar{X}'=(\sigma^{-1}(\dot{X}'))^{\bG}$. Let $\bar{T}_\alpha ' = \bar{T}_\alpha$ for $\alpha \in \blambda \setminus \{ 0 \}$, and let
\[\bT_0 ' = \bT_0 \cup \{ \xi < \omega_1 \st \text{there is no $\alpha \in \bM_\xi \cap \lambda$ with $\bM_\xi \cap \omega_1 \in \bT_\alpha$} \}. \]
For any $\xi < \omega_1$ and any $\alpha < \blambda$, if $\alpha \in \bM_\xi$ and $\bM_\xi \cap \omega_1 \in \bT_\alpha '$, then $\sup ( \bM_\xi \cap \blambda ) \in \bS_\alpha '$, where $\seq{\bS_\alpha '}{\alpha<\blambda} = \sigma^{-1}(\seq{S_\alpha '}{\alpha < \lambda})$. Also, for any $\xi < \omega_1$ there is $\alpha \in \bM_\xi \cap \blambda$ with $\bM_\xi \cap \omega_1 \in \bT_\alpha '$.

Let $M_\xi = \sigma ( \bM_\xi )$ for $\xi<\omega_1$. Note that $\seq{M_\xi}{\xi < \omega_1}$ is an increasing continuous $\in$-chain, and $R= \bigcup_{\xi < \omega_1} M_\xi \cap \lambda = \sigma `` \blambda$. For $\alpha \in R$ let $T_\alpha = \sigma ( \bT_{\balpha} ' ) = \bT_{\balpha}'$. Then $\seq{T_\alpha}{\alpha\in R}$ satisfies the assumption of Lemma \ref{lem:StronglyExactSimultaneousRP} for $\vS '$. Moreover $\vS ' |R = \vS |R$ since $\bigcup_{\xi < \omega_1} M_\xi$ is closed under $f$ and $f^{-1}$. Then $R$ contains a club in $\rho$, and $\rho$ is a strongly exact simultaneous reflection point of $\vS |R$.
\end{proof}

We will see in Subsection \ref{subsec:SeparatingEnhancedSRP-FAfromuDSR} that some form of nonsaturation assumption is needed here: even a stronger form of enhanced \SRP does not even imply a weaker form of diagonal reflection at $\lambda>\omega_2$.

\section{Separation results}
\label{sec:Separations}

In this section, we will show that various forms of the strong reflection principle do not imply certain principles of diagonal reflection.

\subsection{Redoing Larson's argument: separating $\SRP^\FA$ from $\OSR_{\omega_2}$}
\label{subsec:SeparatingSRP-FAfromOSR}

Here, we describe Larson's argument of \cite{Larson:SeparatingSRP}, originally showing that $\SRP$ does not imply $\OSR_{\omega_2}$. We will use some variations of this technique, generalize it for later applications, and strengthen the separation slightly, so we go into some detail.

The basic strategy, as employed by Larson  for the case of stationary set preserving forcing, to show that the fragment of \SRP corresponding to some forcing class $\Gamma$ does not imply some diagonal stationary reflection principle $\Delta$ is to start in a model of $\FA_{\omega_1}(\Gamma)$, force with the natural forcing $\P$ to add a sequence of stationary subsets of some $S^\lambda_\omega$ which exemplifies the failure of $\Delta$, such that there is a further forcing $\dot{\mathbb{T}}$ (usually a kind of club shooting forcing), so that $\P*\dot{\mathbb{T}}$ is $\omega_2$-directed closed. Under mild assumptions, $\FA_{\omega_1}(\Gamma)$ still holds in $\V^{\P*\dot{\mathbb{T}}}$, and as a last step, one argues that in $\V^\P$, the $\Gamma$-fragment of \SRP still holds. We will in fact see that $\Gamma$-$\SRP^\FA$ still holds in Larson's original model (where $\Gamma=\SSP$). The following preservation theorem is useful in showing this in a more general setting. It is slightly different from Larson's original argument, and since $\Gamma$-$\SRP^\FA$ is a forcing axiom, the theorem looks more like a forcing axiom preservation fact, and in fact it fits into the framework of \cite[Theorem 20]{Cox:FAapproachabilitySSR}, but we give here a short, self-contained proof. In our applications of the theorem, the forcing $\Q$ will be of the form $\P_X$.

\begin{thm}
\label{thm:AbstractPreservation}
Let $\Gamma$ be a forcing class closed under restrictions\footnote{$\Gamma$ being closed under restrictions means that if $\mathbb{R}\in\Gamma$ and $r\in\mathbb{R}$, then $\mathbb{R}_{{\le}r}$, the restriction of $\mathbb{R}$ to $\{s\st s\le r\}$, is also in $\Gamma$.} and containing all $\omega_2$-directed closed forcing notions, and assume that $\FA_{\omega_1}(\Gamma)$ holds. Let $\P$ be a forcing notion, $G\sub\P$ generic, and in $\V[G]$, let $\Q=\dot{\Q}^G$ and $\T=\dot{\T}^G$ be forcing notions such that
\begin{enumerate}[label=\textnormal{(\arabic*)}]
  \item In $\V$, there is a dense $\mathbb{U}\sub\P*\dot{\T}$ which is $\omega_2$-directed closed.%
  \footnote{The following slightly weaker condition on $\mathbb{U}$ suffices: whenever $D\sub\mathbb{U}$ is a directed subset of cardinality less than $\omega_2$, there is a $p\in\P$ such that for every $\kla{p',\dot{t}}\in D$, $p\le p'$.}
  \item $\P*(\dot{\Q}\times\dot{\T})\in\Gamma$.\footnote{Here, $\dot{\Q}\times\dot{\T}$ denotes the canonical $\P$-name for the product of $\dot{\Q}$ and $\dot{\T}$.}
\end{enumerate}
Then $\FA_{\omega_1}(\{\Q\})$ holds in $\V[G]$.
\end{thm}

\begin{proof} Assume not.

In $\V[G]$, let $\vec{\Delta}=\dot{\Delta}^G=\seq{\Delta_\xi}{\xi<\omega_1}$ be a sequence of dense open subsets of $\Q$ which cannot be met by a filter in $\Q$. Let $p\in G$ force this.

In $\V$, let, for $\xi<\omega_1$, $D_\xi\sub\R=\P_{{\le}p}*(\dot{\Q}\times\dot{\T})$ be defined by
\[D_\xi=\{\kla{r,\dot{q},\dot{t}}\st r\in\P_{{\le}p},\ \kla{r,\dot{t}}\in\U,\ \text{and}\ r\forces_{\P}\dot{q}\in\dot{\Delta}_{\check{\xi}}\}.\]
Then $D_\xi$ is dense in $\R$, and $\R\in\Gamma$, so we may let $F\sub\R$ be a $\{D_\xi\st\xi<\omega_1\}$-generic filter, $F\in\V$. It is easy to construct $\bar{F}\sub F$ of size $\omega_1$ such that for all $\xi<\omega_1$, $\bar{F}\cap D_\xi\neq\leer$, and such that for all $s,t\in\bar{F}$, there is a $u\in\bar{F}$ such that $u\le s,t$. So $\bar{F}$ is a directed subset of $F$ that generates a $\{D_\xi\st\xi<\omega_1\}$-generic filter in $\mathbb{R}$. Let
\[I=\{\kla{r,\dot{t}}\st\exists\dot{q}\ \kla{r,\dot{q},\dot{t}}\in\bar{F}\}.\]
So $I$ is a directed subset of $\U$, of size at most $\omega_1$. Since $\U$ is $\omega_2$-directed closed, or by our slightly weaker assumption on $\U$, there is an $r^*\in\P$ such that for every $\kla{r,\dot{t}}\in I$, $r^*\le r$.

Let $\tilde{F}=\{\kla{\dot{q},r}\st\exists\dot{t}\ \kla{r,\dot{q},\dot{t}}\in\bar{F}\}$. So $\tilde{F}$ is a $\P$-name for a subset of $\dot{\Q}$.\\

\noindent\emph{Claim:} $r^*\forces_\P``\tilde{F}\ \text{generates a $\dot{\Delta}$-generic filter in $\dot{\Q}$.''}$\\

To see this, let $H$ be $\P$-generic over $\V$ with $r^*\in H$. Let $a,b\in\tilde{F}^H$. Let $\kla{\dot{q}_0,r_0},\kla{\dot{q}_1,r_1}\in\tilde{F}$, where $r_0,r_1\in H$ and $a=\dot{q}_0^H$, $b=\dot{q}_1^H$. Let $\dot{t}_0,\dot{t}_1$ be such that $\kla{r_0,\dot{q}_0,\dot{t}_0},\kla{r_1,\dot{q}_1,\dot{t}_1}\in\bar{F}$. Since $\bar{F}$ is directed, there is a condition $\kla{r_2,\dot{q}_2,\dot{t}_2}\in\bar{F}$, below both $\kla{r_0,\dot{q}_0,\dot{t}_0}$ and $\kla{r_1,\dot{q}_1,\dot{t}_1}$. In particular, $r_2\le r_0,r_1$ and $r_2\forces\dot{q}_2\le\dot{q}_0,\dot{q}_1$. Moreover, since $\kla{r_2,\dot{t}_2}\in I$, it follows that $r^*\le r_2$. Since $r^*\in H$, it follows that $r_2\in H$, and we have that $c=\dot{q}_2^H\le\dot{q}_0^H=a$ and $c\le\dot{q}_1^H=b$, and since $\kla{\dot{q}_2,r_2}\in\tilde{F}$, $c\in\tilde{F}^H$. Thus, $\tilde{F}^H$ generates a filter. Moreover, for $\xi<\omega_1$, there is a condition $\kla{r,\dot{q},\dot{t}}\in\bar{F}\cap D_\xi$, so $r\forces_\P\dot{q}\in\dot{\Delta}_{\check{\xi}}$. Again, since $\kla{r,\dot{t}}\in I$ we have that $r^*\le r$, so $r\in H$. Hence, $\dot{q}^H\in\dot{\Delta}^H_\xi$, and as $\kla{\dot{q},r}\in\tilde{F}$, $\dot{q}^H\in\tilde{F}^H$. So the filter generated by $\tilde{F}^H$ is $\dot{\Delta}^H$-generic.

Since $r^*\le p$, this contradicts what $p$ forced.
\end{proof}

In applications of Theorem \ref{thm:AbstractPreservation}, the forcing $\P$ will add a counterexample to the undesired diagonal reflection principle, that is, a sequence of stationary sets that fails to reflect in a particular way. The forcing $\T$ will shoot a club through the complement of some stationary set, thereby ``killing'' its stationarity, as follows.

\begin{defn}
\label{def:KillingStationarySets}
Let $\lambda$ be an uncountable regular cardinal, and let $A\subseteq\lambda$ be unbounded. The forcing notion $\bb{T}_S$ consists of all closed, bounded subsets $t$ of $\lambda$ such that $t \cap S = \leer$. The ordering is defined by setting, for $t_0, t_1 \in \bb{T}_S$, $t_1 \leq_{\bb{T}_S} t_0$ if
$t_1$ \emph{end-extends} $t_0$, that is, $t_1\cap( \max t_0 + 1 )=t_0$.
\end{defn}

The next lemma provides the basic facts about this type of forcing.

\begin{lem}[{\cite[Lemma 3.8]{Fuchs-LambieHanson:SeparatingDSR}}]
\label{lem:BasicsOnT_S}
Let $\lambda$ be an uncountable regular cardinal, and let $S\subseteq\lambda$ be such that $\lambda\setminus S$ is unbounded in $\lambda$.
\begin{enumerate}[label=\textnormal{(\arabic*)}]
  \item
  \label{item:AddingAClub}
  Let $G$ be $\bb{T}_S$-generic over $V$. Then
  $\bigcup G$ is a club subset of $\lambda$ that is disjoint from $S$.
  \item
  \label{item:TFAE}
  Suppose that $\bb{T}_S$ preserves the fact that $\lambda$ has uncountable cofinality (in our applications, $\bb{T}_S$ will be $\lambda$-distributive, so this will certainly be the case).
  Then, for a set $T\subseteq\lambda$ such that $T\setminus S\subseteq S^\lambda_\omega$, the following are equivalent:
  \begin{enumerate}[label=\textnormal{(\alph*)}]
    \item
    \label{item:T_SpreservesTAlways}
    $\forces_{\bb{T}_S}$ ``$\check{T}$ is stationary in $\check{\lambda}$.''
    \item
    \label{item:T_SpreservesTSometimes}
    There is a $t\in\bb{T}_S$ such that $t\forces_{\bb{T}_S}$ ``$\check{T}$ is stationary in $\check{\lambda}$.''
    \item
    \label{item:TminusSstationary}
    $T\setminus S$ is stationary.
  \end{enumerate}
\end{enumerate}
\end{lem}

Towards characterizing when $\T_S$ preserves the generalized stationarity of a set, the following lemma will be useful. It is basically due to Larson \cite{Larson:SeparatingSRP}; see
\cite[Lemma 4.3]{Fuchs-LambieHanson:SeparatingDSR} for a proof of this exact formulation.

\begin{lem}
\label{lem:WhenT_APreservesGenStationarity}
Let $\lambda>\omega_1$ be regular, $H\supseteq\lambda$ a set, $S\subseteq\lambda$ stationary in $\lambda$ and $X\subseteq[H]^\omega$ stationary in $[H]^\omega$, so that $\lambda\setminus S$ is unbounded in $\lambda$ and $\T_S$ is countably distributive (this is the case, for example, if $S\subseteq S^\lambda_\omega$ and $S^\lambda_\omega \ohne S$ is stationary). Then the following are equivalent:
\begin{enumerate}[label=\textnormal{(\arabic*)}]
\item
\label{item:SminusAisStationary}
$X\setminus\lifting(S,[H]^\omega)$ is stationary.
\item
\label{item:StationarityOfSisPreserved}
$\bb{T}_S$ preserves the stationarity of $X$.
\item
\label{item:SomebodyPreservesStationarityOfS}
There is a condition $p\in\T_S$ that forces that $\check{X}$ is stationary.
\end{enumerate}
\end{lem}

The next lemma is going to enable us to use Theorem \ref{thm:AbstractPreservation}. We omit its proof for now, since it is a special case of Lemma \ref{lem:PreservationOfMediatedProjectiveStationarity}, which we will prove in detail.

\begin{lem}
\label{lem:PreservationOfP_SbeingSSP}
Let $\theta\ge\lambda>\omega_1$ be regular cardinals, and let $S\sub S^\lambda_\omega$ be stationary. Let
$X$ be a stationary subset of $[H_\theta]^\omega$. Suppose that $\T_S$ is countably distributive. If for every stationary $T\sub\omega_1$, $X_T\ohne\lifting(S,[H_\theta]^\omega)$ is stationary, then $\T_S\times\P_X$ is stationary set preserving. Here, we use the notation from Definition \ref{def:ProjectiveStationarity}, writing $X_T=\{x\in X\st x\cap\omega_1\in T\}$. 
\end{lem}

The next theorem is essentially due to Larson \cite{Larson:SeparatingSRP}. Our proof is slightly different, and it separates $\SRP^\FA$ (not just \SRP) from \MM.

\begin{thm}
\label{thm:ALaLarson}
Assume \MM. Let $\P_{\omega_2}$ be the forcing to add a counterexample to $\OSR_{\omega_2}$. Then, in $\V^{\P_{\omega_2}}$, $\SRP^\FA$ holds (while $\OSR_{\omega_2}$ fails).
\end{thm}

\begin{proof}[Proof sketch.]
Let $\vec{S}=\seq{S_\alpha}{\alpha<\omega_2}$ be the generic sequence of stationary subsets of $S^{\omega_2}_\omega$ added by $\P_{\omega_2}$. It does not reflect in the way postulated by $\OSR_{\omega_2}$, and it forms a partition of $S^{\omega_2}_\omega$. The forcing is countably closed, and, letting $\dot{S}_\alpha$ be a canonical $\P_{\omega_2}$-name for $S_\alpha$, the forcing $\P_{\omega_2}*\dot{\T}_{\dot{S}_\alpha}$ has a suborder which is $\omega_2$-directed closed.

Working in $\V[\vec{S}]$, fix a projective stationary set $X\sub[H_\theta]^\omega$, for some regular $\theta\ge\omega_2$. To show that $\FA_{\omega_1}(\{\P_X\})$ holds, it suffices by Theorem \ref{thm:AbstractPreservation} to show that there is an $\alpha<\omega_2$ such that $\P_X\times\T_{S_\alpha}$ is stationary set preserving (as this implies that $\P_{\omega_2}*(\dot{\P}_{\dot{X}}\times\dot{\T}_{\dot{S}_\alpha})$ is stationary set preserving, at least below some condition in $\P_{\omega_2}$ which forces this situation). Since we are in the situation of Lemma \ref{lem:PreservationOfP_SbeingSSP}, it suffices to show the existence of an $\alpha<\omega_2$ such that for every stationary $T\sub\omega_1$, $X_T\ohne\lifting(S_\alpha,[H_\theta]^\omega)$ is stationary. 
 Supposing there were no such $\alpha$, we could find, for each $\alpha<\omega_2$, a stationary $T_\alpha\sub\omega_1$ such that 
 $X_{T_\alpha}\ohne\lifting(S_\alpha,[H_\theta]^\omega)$ is nonstationary.
 Since $\NS_{\omega_1}$ is $\omega_2$-saturated, there are $\alpha_0<\alpha_1<\omega_2$ such that $T=T_{\alpha_0}\cap T_{\alpha_1}$ is stationary. 
 We would then have that the union of $X_{T_{\alpha_0}}\ohne\lifting(S_{\alpha_0},[H_\theta]^\omega)$ and $X_{T_{\alpha_1}}\ohne\lifting(S_{\alpha_1},[H_\theta]^\omega)$ is nonstationary. But this union contains $X_T$, since if $x\in X_T$, then $x\cap\omega_1$ is in both $T_{\alpha_0}$ and $T_{\alpha_1}$, and $\sup(x\cap\omega_2)$  must avoid at least one of $S_{\alpha_0}$ and $S_{\alpha_1}$. Since $X_T$ is stationary, this is a contradiction.
\end{proof}

This previous result is in stark contrast to Theorem \ref{thm:SC-SRP+CHimpliesOSR}, stating that $\infSC$-\SRP $+$ \CH implies even the strongly exact version of $\OSR_{\omega_2}$.

\subsection{Separating enhanced $\SRP^\FA$ from $\uDSR(1,S^\lambda_\omega)$, $\lambda>\omega_2$.}
\label{subsec:SeparatingEnhancedSRP-FAfromuDSR}

We will now work towards showing a stronger separation in the context of diagonal reflection principles at higher cardinals. The goal is to show that
even if the \emph{enhanced} $\SRP^\FA$ holds, $\uDSR(1,S^\lambda_\omega)$ may fail at a regular cardinal $\lambda>\omega_2$. In fact, our result will be stronger than that. It highlights the significance of Theorem \ref{thm:AddingDiamondGivesDiagonalReflection}, which showed that enhanced $\infSC$-$\SRP^\FA+\diamondsuit$ implies strongly exact $\DSR(\omega_1,S^\lambda_\omega)$, for every regular $\lambda>\omega_2$ (for $\lambda=\omega_2$, it already follows from $\infSC$-$\SRP+\CH$, by Theorem \ref{thm:SC-SRP+CHimpliesOSR}).

\begin{obs}
\label{obs:MediatedProjectiveStationarity}
Let $\A$ be a countably distributive, stationary set preserving forcing, and let $\dot{X}$ be an $\A$-name for a stationary subset of $[H_\theta]^\omega$ (in the sense of the forcing extension, where $\theta>\omega_1$ is a regular cardinal.) 
Then $\A*\P_{\dot{X}}$ is stationary set preserving iff for every stationary set $T\sub\omega_1$,
\[\forces_\A\text{``$\dot{X}_{\check{T}}$ is stationary in $\omega_1$.''}\footnote{Here, we use notation introduced in Definition \ref{def:ProjectiveStationarity}.}\]
\end{obs}

\begin{proof}
Fix a stationary set $T\sub\omega_1$. Then $\A*\P_{\dot{X}}$ preserves the stationarity of $T$ iff whenever $I\sub\A$ is generic, $\P_{\dot{X}^I}$ preserves the stationarity of $T$ over $\V[I]$ (since $\A$ is stationary set preserving by assumption). This latter condition is equivalent to saying that in $\V[I]$, $(\dot{X}^I)_T$ is stationary. Indeed, the equivalence between the projective stationarity of some $X\sub[H_\theta]^\omega$ and the preservation of stationary sets by $\P_X$ stated in \cite[Fact 2.9]{Fuchs:CanonicalFragmentsOfSRP} holds locally. One direction of this equivalence can be seen by perusing the proof of \cite[Theorem 1.1]{FengJech:ProjectiveStationarityAndSRP}, showing that if $X_T$ is stationary, then $\P_X$ preserves the stationarity of $T$. The argument of the converse is sketched after that proof; here are the details: suppose $\P_X$ preserves the stationarity of $T$, but assume towards a contradiction that $X_T$ is not stationary. Let $Z\sub[H_\theta]^\omega$ be a club disjoint from $X_T$. If $\seq{M_\xi}{\xi<\omega_1}$ is a $\P_X$-generic sequence, then $\{M_\xi\st\xi<\omega_1\}$ is also club in $[H^\V_\theta]^\omega$. Note that $\P_X$ is countably distributive, so $[H_\theta^\V]^\omega$ does not change, and $Z$ remains club. So $W=Z\cap\{M_\xi\st\xi<\omega_1\}$ is a club subset of $X$, but disjoint from $X_T$. Hence, for all $M\in W$, $M\cap\omega_1\notin T$. Thus, $\{M\cap\omega_1\st M\in W\}$ is a club subset of $\omega_1$ disjoint from $T$, so the stationarity of $T$ is not preserved by $\P_X$, a contradiction.
\end{proof}

The following concept will be useful.

\begin{defn}
\label{def:PotentialStationarity}
Let $\A$ be a forcing notion, $\lambda$ an ordinal, and let $\dot{X}$ be an $\A$-name such that $\forces_\A\text{``$\dot{X}$ is a stationary subset of $[\bigcup\dot{X}]^\omega$, and $\check{\lambda}\sub\bigcup{\dot{X}}$.''}$
Then we say that $\dot{X}$ is \emph{$\A$-potentially stationary above $\lambda$.}	
\end{defn}

Just like Lemma \ref{lem:WhenT_APreservesGenStationarity} provided a criterion for the preservation of generalized stationarity by the forcing to kill a stationary set, the following lemma gives us a criterion for the preservation of $\A$-potential stationarity.

\begin{lem}
\label{lem:PreservationOfPotentialStationarity}
Let $\lambda>\omega_1$ be a regular cardinal, and let $S\sub S^\lambda_\omega$ be stationary. Let $\A$ be a countably distributive forcing, and let $\dot{X}$ be an $\A$-name.
Let $\T_S$ be the forcing to kill the stationarity of $S$. 
If the canonical $\A$-name for $\dot{X}\ohne\lifting(\check{S},[\bigcup\dot{X}]^\omega)$ is $\A$-potentially stationary above $\lambda$, then in any forcing extension by $\T_S$,  $\dot{X}$ is $\A$-potentially stationary above $\lambda$.
\end{lem}

\begin{proof}
We have to show that in any $\T_S$-generic forcing extension, $\A$ forces that $\dot{X}$ is stationary and $\check{\lambda}\sub\bigcup\dot{X}$. Write $\T$ for the forcing $\T_S$.
Turning the order of forcing around, letting $I\sub\A$ be an arbitrary generic filter, and setting $X=\dot{X}^I$ and $H=\bigcup X$, it suffices to show that

\begin{NumberedClaim}
\label{NC:OldT_ApreservesS}
In $\V[I]$, $\T$ forces that $\check{X}$ is stationary in $[\check{H}]^\omega$
\end{NumberedClaim}
since we already know by assumption that $\lambda\sub H$. Note that $\T$ is the forcing to kill the stationarity of $S$ as defined in $\V$, which is different from that forcing as defined in $\V[I]$, that is, $\T=(\T_S)^\V\neq(\T_S)^{\V[I]}$. So we cannot appeal to Lemma \ref{lem:WhenT_APreservesGenStationarity}.

To prove \ref{NC:OldT_ApreservesS}, assume the contrary. So let $\dot{f}\in\V[I]^{\T}$ be a $\T$-name and $p\in\T$ a condition such that
\[p\forces_{\T}^{\V[I]}\text{``$\dot{f}:\check{H}^{{<}\omega}\To\check{H}$ and whenever $x\in[\check{H}]^\omega$ is closed under $\dot{f}$, $x\notin\check{X}$.''}\]
By assumption, and arguing in $\V[I]$, we know that $\tilde{X}=X\ohne\lifting(S,[H]^\omega)$ is stationary. Hence, letting $\chi$ be some sufficiently large regular cardinal, so is $\tilde{X}\lift[H_\chi^{\V[I]}]^\omega$. So let $M$ be a countable set such that $S,\T,p,H,X,\dot{f} \in M$, $\kla{M,{\in}} \prec \kla{H_\chi^{\V[I]},{\in}}$, and $x:=M\cap H\in\tilde{X}$. Let $\alpha=\sup(x\cap\lambda)$. Since $\A$ does not add any new $\omega$-sequences of ordinals, the cofinality of $\lambda$ in $\V[I]$ is still uncountable, so $\alpha<\lambda$, and since $x\notin\lifting(S,[H]^\omega)$, $\alpha\notin S$. Let $g\sub\T\cap M$ be $M$-generic for $\T$ with $p\in g$, generated by a decreasing sequence $\seq{t_n}{n<\omega}$ of conditions in $\T\cap M$ which is defined so that, once an enumeration $\seq{D_n}{n<\omega}$ of the dense subsets of $\T$ which exist in $M$ is fixed, $t_n\in D_n$, for every $n<\omega$. Since $\A$ is countably distributive, the sequence $\seq{t_n}{n<\omega}$ is in $\V$. Note that $\bigcup g=\bigcup_{n<\omega}t_n$ is cofinal in $\alpha$. Hence, as $\alpha\notin S$, the set $q=( \bigcup g )\cup\{\alpha\}$ is a condition in $\T$. Also, by elementarity, $x$ is closed under $\dot{f}^g$, because if $a_0,\ldots,a_{m-1} \in x$, then the set $D$ of conditions which are either incompatible with $p$ or decide the value of $\dot{f}(\check{a}_0,\ldots,\check{a}_{m-1})$ is dense, and it is in $M$, so, since $g$ is generated by $\{t_n\st n<\omega\}$, there is an $n$ such that $r=t_n\in D$. But then $\dot{f}^g(a_0,\ldots,a_{m-1})$ is the unique $b\in H$ such that $r\forces_\T\dot{f}(\check{a}_0,\ldots,\check{a}_{m-1})=\check{b}$. Since $r,\dot{f},\vec{a},\T,H\in M$, it follows by elementarity that $b$ is definable from these parameters in $M$, so $b\in M$, hence $b\in x$, as $b\in H$.

But, if $J\sub\T$ is generic over $\V[I]$, with $q\in J$, then $g\sub J$, and $\dot{f}^J\rest x^{{<}\omega}=\dot{f}^g$. So $x$ is closed under $\dot{f}^J$, and $x\in\tilde{X}$, contradicting what $p$ forced.
\end{proof}

\begin{lem}
\label{lem:PreservationOfMediatedProjectiveStationarity}
Let $\theta\ge\lambda>\omega_1$ be regular cardinals, and let $S\sub S^\lambda_\omega$ be stationary. Let $\A$ be a $\sigma$-closed forcing, and let 
$\dot{X}$ be an $\A$-name such that $\forces_\A\bigcup\dot{X}=H_{\check{\theta}}$ (that is, if $I$ is $\A$-generic, then $\bigcup\dot{X}^I=H_\theta^{\V[I]}$). Let $\T_S$ be the forcing to kill the stationarity of $S$, and suppose that $\T_S$ is countably distributive.
If for every stationary $T\sub\omega_1$,
\[\forces_\A\text{``$\dot{X}_{\check{T}}\ohne\lifting(\check{S},[\bigcup\dot{X}]^\omega)$ is stationary''}\]
then $\T_S\times(\A*\P_{\dot{X}})$ is stationary set preserving.
\end{lem}

\begin{proof}
Let $T\sub\omega_1$ be stationary, and let $G\sub\T_S$ be generic over $\V$.
By our assumption, the canonical $\A$-name for $\dot{X}_T\ohne\lifting(\check{S},[\bigcup{\dot{X}}]^\omega)$ is $\A$-potentially stationary above $\lambda$.

By Lemma \ref{lem:PreservationOfPotentialStationarity}, it follows that in $\V[G]$, $\dot{X}_{\check{T}}$ is $\A$-potentially stationary above $\lambda$. Note that this implies that $T$ is stationary in $\V[G]$.
In any case, letting $I\sub\A$ be generic over $\V[G]$, and setting $X=\dot{X}^I$ and $H=\bigcup X$, we have that in $\V[G][I]$, $X_T$ is stationary in $[H]^\omega$. Note that $H=H_\theta^{\V[I]}$. 

Working in $\V[G][I]$, we can now apply the argument of the proof of Feng \& Jech \cite[Theorem 1.1]{FengJech:ProjectiveStationarityAndSRP}.
Note that $(\P_X)^{\V[I]}=(\P_X)^{\V[G][I]}$, since  $\T_S\times\A$ is countably distributive, as $\A$ is still $\sigma$-closed in $\V[G]$, because $\T_S$ adds no new countable sequences. This implies that $\V[G][I]=\V[I][G]$ has no new $\omega$-sequence of elements of $\V[I]$. So we are justified in just writing $\P_X$ for this poset.

Thus, towards a contradiction, suppose $p\in\P_X$ is a condition and $\dot{C}$ is a $\P_X$-name such that
\[p\forces_{\P_X}^{\V[G][I]}\text{``$\dot{C}\sub\omega_1$ is club and disjoint from $\check{T}$.''}\]
Let $\chi$ be a sufficiently large regular cardinal, and let $M$ be a countable set such that $H,X,T,\dot{C},p \in M$, $\kla{M,{\in}} \prec \kla{H_\chi^{\V[G][I]},{\in}}$, and $x:=M\cap H\in X_T$ (using the fact that $X_T$ is stationary in $[H]^\omega$ in $\V[G][I]$). So $\zeta=x\cap\omega_1=M\cap\omega_1\in T$.

Let $\seq{D_n}{n<\omega}$ enumerate all dense subsets of $\P_X$ that are in $M$, and build a decreasing sequence $\seq{p_n}{n<\omega}$ in $\P_X \cap M$ such that $p_0\le p$ and for all $n<\omega$, $p_n\in D_n$. Let $q'=\bigcup_{n<\omega}p_n$.

Since for every $\xi<\omega_1$, the set of conditions in $\P_X$ of length at least $\xi$ is dense, the length of $q'$ is $\zeta$. Since for every $a\in H$, the set of conditions $r$ such that $a\in\bigcup\ran(r)$ is dense, $\bigcup\ran(q')=x\in X$. Hence, $q=q'\verl\kla{x}$ is a condition in $\P_X$ that extends $p$. Since for every $\xi<\omega_1$, the set of conditions $r$ such that there is an $\eta\ge\xi$ with $r\forces_{\P_X}\text{``$\check{\eta}\in\dot{C}$''}$ is dense, it follows that $q$ forces that $\check{\zeta}$ is a limit point of $\dot{C}$, and hence that $\check{\zeta}$ is in $\dot{C}$. So since $\zeta\in T$, $q$ is an extension of $p$ forcing that $\dot{C}\cap\check{T}\neq\leer$, contradicting what $p$ forced.
\end{proof}

We are now ready to prove the main result of this subsection. It draws and improves on the separation result
obtained by Fuchs \& Lambie-Hanson \cite[Theorem 4.4]{Fuchs-LambieHanson:SeparatingDSR}.

\begin{thm}
\label{thm:EnhancedSRPDoesntImplyuDSR}
Assume $\MM$, and let $\lambda>\omega_2$ be a regular cardinal. Then there is a cofinality preserving forcing extension $W$ of $\V$ in which $\uDSR(1,S^\lambda_\omega)$ fails, yet $\FA_{\omega_1}(\Gamma)$ holds, where, in $W$, $\Gamma$ consists of all stationary set preserving forcing notions of the form
$\A*\P_{\dot{X}}$ such that $\A$ is $\sigma$-closed and $\lambda$-c.c., and $\dot{X}$ is an $\A$-name for a stationary subset of $[H_{\check{\mu}}]^\omega$, for some regular $\mu$ (that is, if $I\sub\A$ is $\A$-generic over $W$, then $\mu>\omega_1$ is regular in $W[I]$ and $\dot{X}^I$ is a stationary subset of $[H_\mu^{W[I]}]^\omega$).
\end{thm}

\noindent\emph{Note:} $\FA_{\omega_1}(\Gamma)$ implies enhanced $\SRP^\FA$, since every forcing of the form $\Add(\omega_1,\theta)$ is countably closed and $\lambda$-c.c., as $\lambda>2^\omega$. So this theorem in particular separates enhanced $\SRP^\FA$ from $\uDSR(1,S^\lambda_\omega)$.

\begin{proof}
Let $\P$ be the $\lambda$-strategically closed forcing used in the proof of \cite[Theorem 3.10]{Fuchs-LambieHanson:SeparatingDSR} to add a counterexample to $\uDSR(1,S^\lambda_\omega)$, by initial segments. It adds a sequence $\seq{S_\alpha}{\alpha<\lambda}$ of pairwise disjoint stationary subsets of $S^{\lambda}_\omega$ such that $S_\alpha\cap(\alpha+1)=\leer$ and such that for every $\rho<\lambda$ of uncountable cofinality, the set $\{\alpha<\rho\st S_\alpha\cap\rho\ \text{is stationary in}\ \rho\}$ is bounded in $\rho$. So this sequence verifies the failure of $\uDSR(1,S^\lambda_\omega)$ in $W=\V[\vS]$, and it remains to show that $\FA_{\omega_1}(\Gamma)$ holds in $W$. For $\alpha<\lambda$, we write
\[S_{{\ge}\alpha}=\bigcup_{\alpha\le\beta<\lambda}S_\beta.\]

Work in $W$.
\begin{NumberedClaim}
\label{NC:ThereAlwaysIsAnAlpha}
Let $\dot{X}$ be $\A$-potentially stationary above $\lambda$, and let $\dot{H}$ be a canonical $\A$-name for the union of $\dot{X}$.
 Then
	\[\forces_\A\text{``$\exists\alpha<\lambda\quad\dot{X}\ohne\lifting(\check{S}_{{\ge}\alpha},[\dot{H}]^\omega)$ is stationary.''}\]
\end{NumberedClaim}	

\begin{proof}[Proof of \ref{NC:ThereAlwaysIsAnAlpha}. ]
If not, then let $I\sub\A$ be generic such that, letting $H=\dot{H}^I$ and $X=\dot{X}^I$, in $W[I]$, for every $\alpha<\lambda$, there is a club $Z_\alpha\sub[H]^\omega$ disjoint from $X\ohne\lifting(S_{{\ge}\alpha},[H]^\omega)$. That is, for all $\alpha<\lambda$, $X\cap Z_\alpha\sub\lifting(S_{{\ge}\alpha},[H]^\omega)$. By the normality of the generalized club filter (see \cite[Lemma 8.23]{ST3}), let $Z\sub[H]^\omega$ be club such that for all $x\in Z$, for all $\alpha\in x\cap\lambda$, $x\in Z_\alpha$. Pick $x\in X\cap Z$, and let $\gamma=\sup(x\cap\lambda)<\lambda$. Then, for all $\alpha\in x\cap\lambda$, $x\in Z_\alpha$, and hence, since $x\in X$, $x\in\lifting(S_{{\ge}\alpha},[H]^\omega)$, which means that $\gamma=\sup(x\cap\lambda)\in S_{{\ge}\alpha}$. Since the $S_\alpha$'s are pairwise disjoint, and since $x\cap\lambda$ is unbounded in $\gamma$, it follows that $\gamma\in S_{{\ge}\gamma}$. But, say $\delta\ge\gamma$ is such that $\gamma\in S_\delta$. Then $\gamma\in S_\delta\cap(\delta +1)=\leer$, a contradiction. \hspace*{\fill}$\qedhere_{\ref{NC:ThereAlwaysIsAnAlpha}}$
\end{proof}

\begin{NumberedClaim}
\label{NC:AFixedThetaAlwaysWorks}
Let $\dot{X}$, $\dot{H}$ be as in \ref{NC:ThereAlwaysIsAnAlpha}. Then
there is an $\alpha<\lambda$ such that
\[\forces_\A\text{``$\dot{X}\ohne\lifting(\check{S}_{{\ge}\alpha},[\dot{H}]^\omega)$ is stationary.''}\]
\end{NumberedClaim}

\begin{proof}[Proof of \ref{NC:AFixedThetaAlwaysWorks}. ]
By \ref{NC:ThereAlwaysIsAnAlpha}, let $\dot{\beta}$ be an $\A$-name such that
\[\forces_\A\text{``$\dot{\beta}<\check{\lambda}$ is least such that $\dot{X}\ohne\lifting(\check{S}_{{\ge}\dot{\beta}},[\dot{H}]^\omega)$ is stationary.''}\]
Let $B=\{\beta<\lambda\st\exists p\in\A\quad p\forces_\A\dot{\beta}=\check{\beta}\}$, and for each $\beta\in B$, let $p_\beta\forces_\A\dot{\beta}=\check{\beta}$. Clearly, if $\beta,\beta ' \in B$ are distinct, then $p_\beta$ and $p_{\beta '}$ are incompatible. So $P=\{p_\beta\st\beta\in P\}$ is an antichain of the same cardinality as $B$. Since $\A$ is $\lambda$-c.c., this cardinality is less than $\lambda$, and since $\lambda$ is regular, $\alpha=\sup B<\lambda$. Then $\alpha$ is as desired, because if $I\sub\A$ is generic, then by \ref{NC:ThereAlwaysIsAnAlpha}, there is a minimal $\beta<\lambda$ such that in $W[I]$, $\dot{X}^I\ohne\lifting(S_{{\ge}\beta},[\dot{H}^I]^\omega)$ is stationary. But then, $\beta\in B$, and hence, $\beta\le\alpha$. So $S_{{\ge}\alpha}\sub S_{{\ge}\beta}$, and hence, in $W[I]$, $\dot{X}^I\ohne\lifting(S_{{\ge}\alpha},[\dot{H}^I]^\omega)$ is stationary as well. \hspace*{\fill}$\qedhere_{\ref{NC:AFixedThetaAlwaysWorks}}$
\end{proof}

Now let $\A*\P_{\dot{X}}\in\Gamma$, where $\A$ and $\dot{X}$ are as described in the statement of the theorem. Let $\dot{H}$ be an $\A$-name such that $\forces_\A\bigcup\dot{X}=\dot{H}$, and let $\mu>\omega_1$ be such that $\forces_\A\dot{H}=H_{\check{\mu}}$.

Assume for now that $\mu\ge\lambda$.

\begin{NumberedClaim}
\label{NC:SomeThetaPreservesProjectiveStationarity}
There is an $\alpha<\lambda$ such that for every stationary $T\sub\omega_1$, (the canonical name for) $\dot{X}_{\check{T}}\ohne\lifting(S_{{\ge}\alpha},[\dot{H}]^\omega)$ is $\A$-potentially stationary.
\end{NumberedClaim}

\begin{proof}[Proof of \ref{NC:SomeThetaPreservesProjectiveStationarity}]
By assumption, $\A*\P_{\dot{X}}$ is stationary set preserving, and hence, so is $\A$. By Observation \ref{obs:MediatedProjectiveStationarity}, this means that for every stationary $T\sub\omega_1$, $\forces_\A\text{``$\dot{X}_{\check{T}}$ is stationary.''}$ Now, fix such a stationary $T$. By \ref{NC:AFixedThetaAlwaysWorks}, there is a $\alpha_T$ such that
\[\forces_\A\text{``$\dot{X}_{\check{T}}\ohne\lifting(\check{A}_{{\ge}\alpha_T},[\dot{H}]^\omega)$ is stationary.''}\]
Since $2^{\omega_1}=\omega_2<\lambda$, $\alpha=\sup\{\alpha_T\st T\sub\omega_1\ \text{is stationary}\}<\lambda$. It follows as before that $\alpha$ is as wished. \hspace*{\fill}$\qedhere_{\ref{NC:SomeThetaPreservesProjectiveStationarity}}$
\end{proof}

We can now finish the proof: let $\alpha$ be as in \ref{NC:SomeThetaPreservesProjectiveStationarity}. In $W$, let $\Q=\A*\P_{\dot{X}}$.
By Lemma \ref{lem:PreservationOfMediatedProjectiveStationarity}, in $W$, $\T_{S_{{\ge}\alpha}}\times\Q$ is stationary set preserving. Now back in $\V$, let $\dot{\T}$ be a $\P$-name for $\T_{S_{{\ge}\alpha}}$, and let $\dot{\Q}$ be a $\P$-name for $\Q$. Then $\P * ( \dot{\Q} \times \dot{\T} )$ is stationary set preserving. Also, by \cite[Claim 3.13]{Fuchs-LambieHanson:SeparatingDSR}, $\P * \dot{\T}$ has an $\omega_2$-directed closed dense subset. So by Theorem \ref{thm:AbstractPreservation}, $\FA_{\omega_1}(\{\Q\})$ holds in $W$.

 We still have to deal with the case that $\mu<\lambda$. This will be a straightforward, but somewhat tedious construction.

 Let $\dot{X}'$ be the canonical $\A$-name for the set of all $x\in\dot{X}\lift[H_{\check{\lambda}}]^\omega$ that are closed under the function $z\mapsto z\cap H_{\check{\mu}}$. That is, whenever $I\sub\A$ is generic over $W$, then
 \begin{ea*}
(\dot{X}')^I=\{x\sub H_\lambda^{W[I]}&\st&\text{$x$ is countable in $W[I]$, $x\cap H^{W[I]}_\mu\in \dot{X}^I$ and}\\ 
&&\text{for all $z\in x$, $z\cap H_\mu^{W[I]}\in x$}\}.
\end{ea*}

 \begin{NumberedClaim}
 \label{NC:A*PS'IsSSP}
 	$\A*\P_{\dot{X}'}$ is stationary set preserving.
 \end{NumberedClaim}

\begin{proof}[Proof of \ref{NC:A*PS'IsSSP}.]

Since $\A*\P_{\dot{X}}$ is stationary set preserving, we know by Observation \ref{obs:MediatedProjectiveStationarity} that $\A$ forces that $\dot{X}_{\check{T}}$ is stationary, for any stationary $T\sub\omega_1$.
So $\A$ forces that $\dot{X}_{\check{T}}\lift[H_{\check{\lambda}}]^\omega$ is stationary.

Let $\dot{Z}$ be a name for the set $\{x\in[H_{\check{\lambda}}]^\omega\st\text{$x$ is closed under $z\mapsto z\cap H_{\check{\mu}}$}\}$. Since this $\dot{Z}$ is forced to be club in $[H_{\check{\lambda}}]^\omega$, $\A$ forces that $\dot{X}_{\check{T}}\lift[H_{\check{\lambda}}]^\omega\cap\dot{Z}$ is stationary. But $\A$ forces that
$\dot{X}_{\check{T}}\lift[H_{\check{\lambda}}]^\omega\cap\dot{Z}
 = (\dot{X}\lift[H_{\check{\lambda}}]^\omega\cap\dot{Z})_{\check{T}}
 =\dot{X}'_{\check{T}}$.
Thus, $\A$ forces that $\dot{X}'_{\check{T}}$ is stationary, for every stationary $T\sub\omega_1$. By Observation \ref{obs:MediatedProjectiveStationarity} again, this means that $\A*\P_{\dot{X}'}$ is stationary set preserving.
  \hspace*{\fill}$\qedhere_{\ref{NC:A*PS'IsSSP}}$\end{proof}

Hence, by the above argument, we have:

\begin{NumberedClaim}
\label{NC:FAforA*PS'}
$\FA_{\omega_1}(\{\A*\P_{\dot{X}'}\})$ holds.
\end{NumberedClaim}

We will argue that this implies $\FA_{\omega_1}(\{\A*\P_{\dot{X}}\})$. To this end, we devise a way to lift an arbitrary dense set $D$ of conditions in $\A*\P_{\dot{X}}$ to a dense subset $D'$ of  $\A*\P_{\dot{X}'}$.

First, if $\kla{a,\dot{p}'}\in\A*\P_{\dot{X}'}$, then whenever $I\sub\A$ is generic over $W$ with $a\in I$, $(\dot{p}')^I$ is an increasing continuous $\in$-sequence $\vx'=\seq{x'_\xi}{\xi\le\zeta}$, for some $\zeta<\omega_1$, with $x'_\xi\in X'=(\dot{X}')^I$, and we can consider in $W[I]$ the sequence $\vx=\seq{x_\xi}{\xi\le\zeta}$ defined by $x_\xi=x'_\xi\cap H_\mu^{W[I]}$, for $\xi\le\zeta$. Then $\vx$ is an increasing continuous $\in$-chain: being increasing continuous is obvious, and for $\xi\leq\zeta$, we have $x'_\xi\in x'_{\xi+1}$, so since $x'_{\xi+1}\in X'$ and hence $x'_{\xi+1}$ is closed under the function $z\mapsto z\cap H_\mu^{W[I]}$, $x_\xi=x'_\xi\cap H_\mu^{W[I]}\in x'_{\xi+1}$. But since $x_\xi$ is a countable subset of $H_\mu^{W[I]}$, and since $\mu$ is a regular uncountable cardinal in $W[I]$, $x_\xi\in H_\mu^{W[I]}$. Thus, $x_\xi\in x'_{\xi+1}\cap H_\mu^{W[I]}=x_{\xi+1}$. Clearly, each $x_\xi$ is in $X=\dot{X}^I$, and so, $\vec{x}\in\P_X$. This construction can be carried out in $W[I]$ whenever $a\in I$, so there is a canonical name for $\vec{x}$, denote it $\kla{a,\dot{p}'}\sqcap H_\mu$.

Returning to our dense set $D\sub\A*\P_{\dot{X}}$, we can now define
\[D'=\{\kla{a,\dot{p}'}\in\A*\P_{\dot{X}'}\st\exists\dot{p}\ (\kla{a,\dot{p}}\in D\ \land\ a\forces_\A\dot{p}=\kla{a,\dot{p}'}\sqcap H_\mu)\}.\]

\begin{NumberedClaim}
\label{NC:D'isDense}
$D'\sub\A*\P_{\dot{X}'}$ is dense.
\end{NumberedClaim}

\begin{proof}[Proof of \ref{NC:D'isDense}]
Let $\kla{a,\dot{p}'}\in\A*\P_{\dot{X}'}$ be given. Let $\dot{p}=\kla{a,\dot{p}'}\sqcap H_\mu$. So $\kla{a,\dot{p}}\in\A*\P_{\dot{X}}$. Since $D$ is dense in $\A*\P_{\dot{X}}$, we can pick $\kla{b,\dot{q}}\le_{\A*\P_{\dot{X}}}\kla{a,\dot{p}}$ in $D$.

Let $I\sub\A$ be generic, with $b\in I$. Let $\vx'=\seq{x'_\xi}{\xi\le\zeta}=(\dot{p}')^I$, $\vx=\seq{x_\xi}{\xi\le\zeta} = \dot{p}^I$ and $\vy=\seq{y_\xi}{\xi\le\zeta^*} = \dot{q}^I$. Define $\vy'=\seq{y'_\xi}{\xi\le\zeta^*}$ by letting $y'_\xi=x'_\xi$ for $\xi\le\zeta$.
To define $y'_\xi$ for $\zeta<\xi\le\zeta^*$, suppose $y'_\eta$ has already been defined for $\eta<\xi$. If $\xi$ is a limit ordinal, then let $y'_\xi=\bigcup_{\eta<\xi}y'_\eta$. Now suppose $\xi=\eta+1$. Then let $y'_\xi=y_\xi\cup y'_\eta \cup \{ y'_\eta \}$.

Let $X=\dot{X}^I$, $X'=(\dot{X}')^I$, $H=H_\mu^{W[I]}$. We claim that $\vec{y}'\in\P_{X'}$. $\vy '$ is increasing continuous by definition. It is an $\in$-chain: let $\xi\leq\zeta^*$. If $\xi<\zeta$, then $y'_\xi=x'_\xi\in x'_{\xi+1}=y'_{\xi+1}$. If $\zeta\leq\xi<\zeta^*$ then $y'_\xi\in y'_{\xi+1}$ by definition.

Further, for $\xi\leq\zeta^*$, $y'_\xi\cap H=y_\xi$: this is clear for $\xi\le\zeta$ since then $y'_\xi\cap H=x'_\xi\cap H=x_\xi=y_\xi$. If $\zeta<\xi\le\zeta^*$, then we proceed by induction. If $\xi$ is a limit ordinal then it is clear by the continuity of $\vy$. So suppose $\xi=\eta +1$. Then $y'_\xi\cap H=(y_\xi\cup y'_\eta \cup \{ y'_\eta \})\cap H$. Now there are two possibilities. First, it could be that $y'_\eta\in H$. In this case $y'_\eta=y_\eta$, and it follows that $y'_\xi=y_\xi\cup y'_\eta \cup\{ y'_\eta \}=y_\xi\cup y_\eta \cup \{ y_\eta \} = y_\xi$ as $y_\eta \cup \{ y_\eta \} \subseteq y_\xi$. So $y'_\xi \cap H = y_\xi$. Otherwise, $y'_\eta\notin H$. In this case, we have $y'_\xi\cap H=(y_\xi\cup y'_\eta \cup \{ y'_\eta \})\cap H= y_\xi \cup y_\eta = y_\xi$, as claimed.

So, since $y_\xi\in X$, $y'_\xi\in X\lift([H_\lambda]^\omega)^{W[I]}$.
Moreover, $y'_\xi$ is closed under $z\mapsto z\cap H$: we proceed by induction. This is clear for $\xi\le\zeta$, since then, $y'_\xi=x'_\xi$. It is also clear for limit $\xi$. So suppose $\xi=\eta +1$. Let $z\in y'_\xi=y_\xi\cup y'_\eta \cup \{ y'_\eta \}$. If $z\in y_\xi$, then $z\cap H=z\in y_\xi \subseteq y'_\xi$. If $z \in y'_\eta$, then, by the induction hypothesis, $z\cap H \in y'_\eta \subseteq y'_\xi$. If $z = y'_\eta$, then $z \cap H = y'_\eta\cap H=y_\eta\in y_\xi \subseteq y'_\xi$.

Hence, $\vy'\in\P_{X'}$, and obviously, $\vy'\le_{\P_{X'}}\vx'$.

The construction of $\vy'$ can be carried out in $W[I]$ whenever $b\in I$, so there is an $\A$-name for $\vy'$, call it $\dot{q}'$. Thus, we have that $b\forces_\A\dot{q}'\le_{\P_{\dot{X}'}}\dot{p}'$, so that $\kla{b,\dot{q}'}\le_{\A*\P_{\dot{X}'}}\kla{a,\dot{p}'}$. Further, $\kla{b,\dot{q}'}\in D'$, as witnessed by $\dot{q}$, since $\kla{b,\dot{q}}\in D$, and for $\xi\le\zeta^*$, $y'_\xi\cap H=y_\xi$. \hfill$\qedhere_{\ref{NC:D'isDense}}$
\end{proof}

Now, let $\Delta$ be a collection of $\omega_1$ dense subsets of $\A*\P_{\dot{X}}$, and let $\Delta'=\{D'\st D\in\Delta\}$. Then $\Delta'$ is a collection of $\omega_1$ dense subsets of $\A*\P_{\dot{X}'}$, so by \ref{NC:FAforA*PS'}, there is a $\Delta'$-generic filter $F'\sub\A*\P_{\dot{X}'}$. Let
\[
F=\{\kla{a,\dot{p}}
\st\exists\dot{p}'\ (\kla{a,\dot{p}'}\in F'
\land a\forces_\A \dot{p}=\kla{a,\dot{p}'}\sqcap H_\mu\}.
\]
It is not hard to see that $F$ is a filter in $\A*\P_{\dot{X}}$. Moreover, it is $\Delta$-generic, because given $D\in\Delta$, let $\kla{a,\dot{p}'}\in D'\cap F'$. By definition of $D'$, let $\dot{p}$ be such that $\kla{a,\dot{p}}\in D$ and $a$ forces that $\dot{p}=\kla{a,\dot{p}'}\sqcap H_\mu$. Then $\kla{a,\dot{p}}\in F$, as witnessed by $\dot{p}'$.
\end{proof}

%

\subsection{When $\NS_{\omega_1}$ is $\omega_1$-dense: separating $\SRP$ from $\uDSR(1,S^{\omega_2}_\omega)$}
\label{subsec:NSomega1-dense}

In this subsection, we will see that boosting the saturation properties of $\NS_{\omega_1}$ to being $\omega_1$-dense (meaning that there is a collection $\mathcal{D}$ of $\omega_1$ many stationary subsets of $\omega_1$ such that for every set $S$ stationary in $\omega_1$, there is a $T\in\mathcal{D}$ such that $T\ohne S$ is not stationary) allows us to construct a model where $\SRP$ holds, yet an even weaker diagonal reflection principle than $\OSR_{\omega_2}$ fails. Let us point out that the $\omega_1$-density of $\NS_{\omega_1}$ contradicts \MM. This can be seen by referring to Larson \cite{Larson:SeparatingSRP}: it is noted there that \MM implies the simultaneous reflection principle for generalized stationary sets denoted $\mathsf{SR}^*_{\omega_1}$ (see \cite[Def.~2.4 and the following remark]{Larson:SeparatingSRP}), and by \cite[Theorem 6.4]{Larson:SeparatingSRP}, $\mathsf{SR}^*_{\omega_1}$ implies that the nonstationary ideal on $\omega_1$ is not $\omega_1$-dense. However, by recent work of Lietz, \SRP is consistent with $\NS_{\omega_1}$ being $\omega_1$-dense, assuming the consistency of sufficiently large cardinals. For our application, we want \SRP to hold indestructibly, and we begin by observing that this indestructibility holds in Lietz's model.

\begin{defn}[{\cite[P.~97 and Def.~2.1, 2.2]{Lietz:NSdense}}]
The principle $\diamondsuit(\omega_1^{{<}\omega})$ states%
\footnote{This is a strengthening due to Lietz of a principle of the same name introduced earlier by Woodin \cite[Def.~6.37]{Woodin:ADforcingAxiomsNonstationaryIdeal}.}
that there exists a function $f:\omega_1\To H_{\omega_1}$ such that for every $\zeta<\omega_1$, $f(\zeta)$ is a filter in $\Col{\omega,\zeta}$ and such that for every $b\in\Col{\omega,\omega_1}$ and every regular $\theta>\omega_1$, letting $\calN = \kla{H_\theta , {\in},f}$, the set
\begin{ea*}
X_{\theta,b,f}=\{M\in[H_\theta]^\omega&\!\st\!& \calN | M \prec \calN ,\ b\in f(M\cap\omega_1)\\ &&\text{and}\ f(M\cap\omega_1)\ \text{is}\ (M,\Col{\omega,M\cap\omega_1})\text{-generic}\}
\end{ea*}%
is stationary in $[H_\theta]^\omega$. In this situation, we say that $f$ witnesses $\diamondsuit(\omega_1^{{<}\omega})$, and we say that a forcing $\P$ is $f$-preserving if after forcing with $\P$, $f$ still witnesses $\diamondsuit(\omega_1^{{<}\omega})$.
\end{defn}

\begin{defn}[{\cite[Def.~7.1]{Lietz:NSdense}}]
\label{def:QM}
$\mathsf{QM}$ is the statement saying that there is a function $f$ which witnesses $\diamondsuit(\omega_1^{{<}\omega})$ such that the forcing axiom $\FA_{\omega_1}(\Gamma)$ for the class $\Gamma$ consisting of all $f$-preserving forcing notions holds.
\end{defn}

\begin{thm}[Lietz]
\label{thm:Lietz}
Assuming the consistency of an inaccessible limit of supercompact cardinals, $\mathsf{QM}$ is consistent. $\mathsf{QM}$ implies $\SRP$ and that $\NS_{\omega_1}$ has density $\omega_1$.	
\end{thm}

\begin{obs}
\label{obs:QMpreservation}
$\mathsf{QM}$, if true, is preserved by $\omega_2$-directed closed forcing. 	
\end{obs}

\begin{proof}
Cox \cite[Theorem 44]{Cox:FAapproachabilitySSR} states that for any forcing class $\Gamma$ which is closed under restrictions and under 2-step iterations, $\FA_{\omega_1}(\Gamma)$ is preserved by $\omega_2$-directed closed forcing. The class of $f$-preserving forcing notions obviously has these two properties, but actually, the argument of \cite[Theorem 44]{Cox:FAapproachabilitySSR} also needs that the class in question contains every $\omega_2$-directed closed forcing notion, and that it is closed under forcing equivalence. Since it is obvious that the class of $f$-preserving forcing is closed under forcing equivalence, let us quickly check that it contains the $\omega_2$-directed closed forcings. In fact, it contains all $\omega_2$-closed forcings.

To see this, let $f$ witness that $\diamondsuit(\omega_1^{{<}\omega})$ holds, let $\P$ be $\omega_2$-closed, and let $G$ be $\P$-generic over $\V$. Fix a condition $b\in\Col{\omega,\omega_1}$ and a regular cardinal $\theta>\omega_1$ in $\V[G]$. We have to show that $(X_{\theta,b,f})^{\V[G]}$ is stationary in $[H_\theta^{\V[G]}]^\omega$ (in $\V[G]$). But since $f$ is a $\diamondsuit(\omega_1^{{<}\omega})$-sequence, we know that $\bX=(X_{\theta,b,f})^V$ is stationary in $[H_\theta^\V]^\omega$ (in $\V$). Since $\P$ is proper, it preserves this property, so $\bX$ is a stationary subset of $[H_\theta^\V]^\omega$ also in $\V[G]$. So, working in $\V[G]$,
\[X=\bX\lift[H_\theta^{\V[G]}]^\omega=\{N\in[H_\theta^{\V[G]}]^\omega\st N\cap H_\theta^\V\in\bX\}\]
is stationary in $[H_\theta^{\V[G]}]^\omega$. So the set $X'=\{N\in X \st \kla{N,{\in}}\prec \kla{H_\theta^{\V[G]},{\in}}\}$ is stationary, too.

But $X'\sub(X_{\theta,b,f})^{\V[G]}$: let $N\in X'$ and $M=N\cap H_\theta^\V$. Since $M\in\bX$, $f\in M\sub N$. Further, let $\zeta =M\cap\omega_1=N\cap\omega_1$. Since $M\in\bX$, $b\in f(\zeta)=f(N\cap\omega_1)$. Finally, let's check that $f(\zeta)$ is $(N,\Col{\omega,\omega_1})$-generic. Working in $\V[G]$, let $D\sub\Col{\omega,\omega_1}$ be a dense subset, $D\in N$. Since $\P$ is $\omega_2$-closed, $D\in\V$. So actually, $D\in N\cap H_\theta^\V =M$. So, since $f(\zeta)$ is $(M,\Col{\omega,\delta})$-generic, there is a $c\in f(\zeta)\cap D\cap M$. So $c\in f(\zeta)\cap D \cap N$.

Thus, $X'\sub(X_{\theta,b,f})^{\V[G]}$, and hence, the latter is stationary in $\V[G]$.
\end{proof}

\begin{obs}
\label{obs:QMimpliesSRPFA}
$\mathsf{QM}$ implies $\SRP^\FA$.	
\end{obs}

\begin{proof}[Proof sketch. ]
By $\mathsf{QM}$, let $f$ be a $\diamondsuit(\omega_1^{{<}\omega})$ witness such that the forcing axiom for all $f$-preserving forcing notions holds. We will have to use some concepts introduced in \cite{Lietz:NSdense} without explaining their meaning. Thus, by \cite[Lemma 7.5]{Lietz:NSdense}), $\MM(f)$ holds, and this  implies $\mathsf{SPFA}(f)$ (they are actually equivalent, by \cite[Thm.~3.67(a)]{Lietz:NSdense}). Now, Lietz \cite[Lemma 3.69]{Lietz:NSdense} shows that $\mathsf{SPFA}(f)$ implies \SRP. We will argue that actually, $\mathsf{SPFA}(f)$ implies $\SRP^\FA$ (so our observation is actually slightly stronger than stated).

To this end, let $X\sub[H_\theta]^\omega$ be projective stationary. We claim that $\P_X$ is $f$-semiproper (see \cite[Def.~3.31]{Lietz:NSdense}). We already know by \cite[Lemma 3.69]{Lietz:NSdense} that \SRP holds. This implies the reflection principle $\mathsf{RP}$ of \cite[Thm 37.23]{ST3}. This principle $\mathsf{RP}$, in turn, implies the statement $(\dagger)$ that every stationary set preserving forcing is semiproper; see \cite[Ex.~37.13]{ST3}. But $(\dagger)$ implies that being projective stationary is equivalent to being  spanning (\cite[Thm. 4.4]{FengJechZapletal:StructureOfStationarySets}). Thus, our set $X$ is spanning, meaning that for every $\lambda>2^\theta$ and every club $Z\sub[H_\lambda]^\omega$ there is a club $Y\sub[H_\lambda]^\omega$ such that for every $M\in Y$ there is a $N\in Z$ such that $N\cap H_\theta\in X$, $M\sub N$ and $M\cap\omega_1=N\cap\omega_1$. 

To show that $\P_X$ is $f$-semiproper, closely following the proof of \cite[Lemma 3.69]{Lietz:NSdense}, we fix a condition $p\in\P_X$ and some regular $\lambda>2^\theta$, and we let $Z$ be the club of all $M \in [H_\lambda]^\omega$ with $\kla{M,{\in},\Delta \rest M} \prec \kla{H_\lambda,{\in},\Delta}$, where $\Delta$ is a well-ordering of $H_\lambda$. Let $Y$ be a club subset of $[H_\lambda]^\omega$ as in the definition of being spanning. Let $M\in Y$ be $f$-slim, with $\P_X,f,p\in M$. We have to find a $q\le p$ which is $(M,\P_X,f)$-semigeneric.

This can be done as in the proof of \cite[Lemma 3.69]{Lietz:NSdense}: let $N\in Z$ be such that $M\cap\omega_1=N\cap\omega_1$, $M\sub N$ and $N\cap H_\theta\in X$. Letting $\zeta=M\cap\omega_1=N\cap\omega_1$, it follows by \cite[Thm.~3.30]{Lietz:NSdense} that $f(\zeta)$ is $\Col{\omega,\zeta}$-generic over $\bN$, the Mostowski collapse of $N$. Let $\sigma : \bN \To N$ be the inverse of the Mostowski collapsing map. Let $\seq{\bD_n}{n<\omega}$ be an enumeration of all dense subsets of $\sigma^{-1} ( \P_X )$ in $\bN[ f( \zeta ) ]$, and take a decreasing sequence $\seq{\bp_n}{n<\omega}$ in $\sigma^{-1} ( \P_X )$ with $\bp_n \in \bD_n$ for all $n < \omega$. Let $p_n = \sigma ( \bp_n )$ for $n<\omega$. Then $\bigcup_{n<\omega}p_n$ is of the form $\seq{N_\xi}{\xi<\zeta}$, and $N_\zeta:=\bigcup_{\xi<\zeta}N_\xi=N\cap H_\theta\in X$. So $q=\seq{N_\xi}{\xi\le\zeta}$ is a condition in $\P_X$, and it follows as outlined at the end of the proof of  \cite[Lemma 3.69]{Lietz:NSdense} that $q\le p$ is $(M,\P_X,f)$-semigeneric. 

Thus, by $\SPFA(f)$, $\FA_{\omega_1}(\{\P_X\})$ holds.
\end{proof}

Putting together Theorem \ref{thm:Lietz}, Observation \ref{obs:QMpreservation} and Observation \ref{obs:QMimpliesSRPFA}, we obtain the following.

\begin{cor}
\label{cor:AssumptionConsistent}
If it is consistent that there is an inaccessible limit of supercompact cardinals, then it is consistent that the density of $\NS_{\omega_1}$ is $\omega_1$, and that $\SRP^\FA$ holds in $\V$ as well as in every forcing extension of $\V$ by an $\omega_2$-directed closed forcing notion.
\end{cor}

Thus, the assumptions of the result Fuchs \& Lambic-Hanson \cite[Cor.~4.7]{Fuchs-LambieHanson:SeparatingDSR} are consistent, assuming an inaccessible limit of supercompact cardinals, allowing us to construct a model of set theory where \SRP holds but $\uDSR(1,S^{\omega_2}_\omega)$ fails. We will provide a proof that corrects a small blunder in the original argument and strengthens the conclusion. 

\begin{thm}
\label{thm:INSdenseuDSRfailsAtOmega2}
If the density of $\NS_{\omega_1}$ is $\omega_1$ and $\SRP^\FA$ holds and continues to hold in any forcing extension by an $\omega_2$-directed closed forcing notion, then there is an $\omega_2$-strategically closed forcing notion which produces forcing extensions where $\SRP^\FA$ holds yet $\uDSR(1,S^{\omega_2}_\omega)$ fails.
\end{thm}

\noindent\emph{Note:} As mentioned before, the principles $\OSR_{\omega_2}$, $\uDSR(\omega_1,S^{\omega_2}_\omega)$ and $\DSR(\omega_1,S^{\omega_2}_\omega)$ are equivalent; see Fuchs \& Lambie-Hanson \cite[Lemma 2.4]{Fuchs-LambieHanson:SeparatingDSR}. Furthermore, \cite[Thm.~3.15]{Fuchs-LambieHanson:SeparatingDSR} shows that $\uDSR({<}\omega_1,S^{\omega_2}_\omega)$ does not imply $\sDSR(1,S^{\omega_2}_\omega)$, and \cite[Cor. 3.20]{Fuchs-LambieHanson:SeparatingDSR} shows that $\sDSR({<}\omega,S^{\omega_2}_\omega)$ does not imply $\DSR(1,S^{\omega_2}_\omega)$ (and the latter is a consequence of $\OSR_{\omega_2})$. Hence, this theorem provides a stronger separation result than Theorem \ref{thm:ALaLarson}.

\begin{proof}
By assumption, we may fix an $\omega_1$-sized collection $\calT$ of stationary subsets of $\omega_1$ which is dense in the stationary subsets of $\omega_1$, that is, for every stationary $U\subseteq\omega_1$, there is a $T\in\calT$ such that $T\ohne U$ is nonstationary. In the original proof, it was carelessly stated that an $\omega_1$-sized collection $\calT$ of stationary subsets of $\omega_1$ can be found such that for every stationary $U\sub\omega_1$ there is an $T\in\calT$ such that $T\sub U$, but this is easily seen to be impossible - in fact, it is easy to see that for any collection of $\omega_1$ many unbounded subsets of $\omega_1$, there is a club subset of $\omega_1$ that is not a superset of any set from the given collection. But the proof goes through without difficulty, as we shall see.

Let $\P$ be the $\omega_2$-strategically closed forcing used in the proof of \cite[Theorem 3.10]{Fuchs-LambieHanson:SeparatingDSR} to add a counterexample to $\uDSR(1,S^{\omega_2}_\omega)$, by initial segments. It is the same forcing we already used in the proof of Theorem \ref{thm:EnhancedSRPDoesntImplyuDSR}, but this time at $\omega_2$. It adds a sequence $\seq{S_\alpha}{\alpha<\omega_2}$ of pairwise disjoint stationary subsets of $S^{\omega_2}_\omega$ such that $S_\alpha\cap(\alpha+1)=\leer$ and such that for every $\rho<\omega_2$ of uncountable cofinality, the set $\{\alpha<\rho\st S_\alpha\cap\rho\ \text{is stationary in}\ \rho\}$ is bounded in $\rho$. It remains to show that $\SRP^\FA$ holds in $\V[\vS]$.

Working in $\V[\vec{S}]$, define $S_{{\ge}\alpha}=\bigcup_{\alpha\le\beta<\omega_2}S_\beta$, and let $\T_\alpha$ be the forcing $\T_{S_{{\ge}\alpha}}$ of Definition \ref{def:KillingStationarySets}, shooting a club through the complement of $S_{{\ge}\alpha}$, thus killing the stationarity of $S_{{\ge}\alpha}$. Letting $\dot{\T}_\alpha$ be a canonical $\P$-name for $\T_\alpha$,  \cite[Theorem 3.10, Claim 3.13]{Fuchs-LambieHanson:SeparatingDSR} shows that $\P*\dot{\T}_\alpha$ contains an $\omega_2$-directed closed subset of size $\omega_2$ (using that under our assumptions, $\omega_2^{\omega_1}=\omega_2$.) Moreover, by \cite[Theorem 4.4, Claim 4.5]{Fuchs-LambieHanson:SeparatingDSR}, we have that in $V[\vS]$, for every stationary $X\subseteq [H_\theta]^\omega$, where $\theta\ge\omega_2$, there is an $\alpha<\omega_2$ such that $\T_\alpha$ preserves the stationarity of $X$. Note also that since $\P$ does not add new subsets of $\omega_1$, $\calT$ is still dense in the stationary subsets of $\omega_1$ from the point of view of $\V[\vS]$. 

We use these facts in order to show that $\SRP^\FA$ holds in $\V[\vS]$. To this end, working in $\V[\vS]$, let $\theta\ge\omega_2$ be regular, and let $X\sub[H_\theta]^\omega$ be projective stationary. We will find an $\alpha^*<\omega_2$ such that in $\V[\vec{S}]$, $\T_{\alpha^*}\times\P_X$ is stationary set preserving. By Lemma \ref{lem:WhenT_APreservesGenStationarity} and \ref{lem:PreservationOfP_SbeingSSP}, for this, it suffices to find an $\alpha^*$ such that for every stationary $U\sub\omega_1$, $\T_{\alpha^*}$ preserves the stationarity of $X_U=\{x\in X\st x\cap\omega_1\in U\}$. 
To find such an $\alpha^*$, first, for every stationary $U\sub\omega_1$, $X_U$ is stationary by the projective stationarity of $X$, so by the above property, we can let $\alpha_U<\omega_2$ be such that $\T_{\alpha_U}$ preserves the stationarity of $X_U$, and then set $\alpha^*=\sup_{T\in\calT}\alpha_T<\omega_2$ (since $\card{\calT}=\omega_1$.) To see that $\alpha^*$ is as wished, let $U\sub\omega_1$ be an arbitrary stationary set. Let $T\in\calT$ be such that $T\ohne U$ is not stationary. Since $\alpha_T\le\alpha^*$, $S_{{\ge}\alpha^*}\sub S_{{\ge}\alpha_T}$, and hence, \[X_T\ohne\lifting(S_{{\ge}\alpha_T},[H_\theta]^\omega)\sub X_T\ohne\lifting(S_{{\ge}\alpha^*},[H_\theta]^\omega).\]
Since $\T_{\alpha_T}$ preserves the stationarity of $X_T$, the former set is stationary by Lemma \ref{lem:WhenT_APreservesGenStationarity}, hence so is the latter. Now, since $T\ohne U$ is not stationary in $\omega_1$, pick a club $C\sub\omega_1$ disjoint from $T\ohne U$, so that $T\cap C\sub U$. Let $Z^*$ be the set of all $M\in[H_\theta]^\omega$ with $C\in M$ and $\kla{M,{\in}}\prec\kla{H_\theta,{\in}}$. Then $Z^*$ is a club subset of $[H_\theta]^\omega$, and for all $M\in Z^*$, $M\cap\omega_1\in C$, because $M\cap\omega_1$ is a limit point of $C$. Thus,
\[(X_T\ohne\lifting(S_{{\ge}\alpha^*},[H_\theta]^\omega))\cap Z^*\sub X_U\ohne\lifting(S_{{\ge}\alpha^*},[H_\theta]^\omega),\]
so again, since the former set is stationary, so is the latter. By Lemma \ref{lem:WhenT_APreservesGenStationarity} again, this means that $\T_{\alpha^*}$ preserves the stationarity of $X_U$, and so, $\alpha^*$ is as wished. It follows now from Theorem \ref{thm:AbstractPreservation} that $\FA_{\omega_1}(\{\P_X\})$ holds in $\V[\vec{S}]$. Since $X$ was an arbitrary projective stationary set, this shows that $\SRP^\FA$ holds in $\V[\vec{S}]$, as claimed.
\end{proof}

To emphasize, it is the regular version of $\uDSR(1,S^{\omega_2}_\omega)$ which fails in the $\omega_2$-strategically closed forcing extension of the corollary, not merely its exact version.

\begin{question}
Does $\SRP + $``$\NS_{\omega_1}$ is not $\omega_1$-dense'' imply (exact) $\uDSR(1,S^{\omega_2}_\omega)$?	
\end{question}

\section{Towards separating $\infSC$-\SRP from diagonal reflection principles}

In this section, we will analyze what happens when one tries to imitate the Larson argument to establish limitations to the consequences \infSC-\SRP has in terms of diagonal stationary reflection principles. The main tool in proving that $\SRP^\FA$ is preserved when forcing over a model of \MM to add a counterexample to $\OSR_{\omega_2}$ is the following consequence of Lemma \ref{lem:WhenT_APreservesGenStationarity} and \ref{lem:PreservationOfP_SbeingSSP}

\begin{lem}
\label{lem:whent_apreversesprojstationarity}
Let $\theta\ge\lambda>\omega_1$ be regular cardinals, $S\subseteq \lambda$ a stationary set of ordinals and $X\subseteq[H_\theta]^\omega$ projective stationary, such that $\lambda\setminus S$ is unbounded in $\lambda$ and $\T_S$ is countably distributive. 
Then the following are equivalent:
\begin{enumerate}[label=\textnormal{(\arabic*)}]
\item
\label{item:SminusAisProjStationary}
$X\setminus\lifting(S,[H_\theta]^\omega)$ is projective stationary.
\item
\label{item:ProductIsSSP}
$\T_S\times\P_X$ is stationary set preserving.
\end{enumerate}
\end{lem}

\begin{proof}

\ref{item:SminusAisProjStationary}$\implies$\ref{item:ProductIsSSP}: Assumption  \ref{item:SminusAisProjStationary} implies that for every stationary $T\sub\omega_1$, $X_T\ohne\lifting(S,[H_\theta]^\omega)$ is stationary. Thus, by Lemma \ref{lem:PreservationOfP_SbeingSSP}, $\T_S\times\P_X$ is stationary set preserving.

\ref{item:ProductIsSSP}$\implies$\ref{item:SminusAisProjStationary}: Fixing a stationary $T\sub\omega_1$, we must show that $X_T\ohne\lifting(S,[H_\theta]^\omega)$ is stationary. Suppose this is not the case. Then let $Z\sub[H_\theta]^\omega$ be club and disjoint from this set, that is, $X_T\cap Z\sub\lifting(S,[H_\theta]^\omega)$. 

Let $\vM$ be a $\P_X$-generic sequence over $\V$. By Fact \ref{fact:P_SisCountablyDistributiveAndExhaustive}, it is an increasing continuous $\in$-chain of elements of $X$ whose union is $H_\theta^\V$, and in $\V[\vM]$, the cofinality of $\lambda$ is $\omega_1$.
Now let $I$ be $\V[\vM]$-generic for $\T_S$. Let $B=\bigcup I$, so $B$ is the club subset of $\lambda\ohne S$ added by $I$.
Note that since $\P_X\times\bb{T}_S$ is stationary set preserving, $\omega_1^\V$ is still a cardinal in $\V[\vM][I]$, and hence, $\lambda$ still has uncountable cofinality in $\V[\vM][I]$. 

Thus, letting
\[C_0=\{\xi<\omega_1\st\sup(M_\xi\cap\lambda)\in B\},\]
$C_0$ is club in $\omega_1$. Further, define
\[C_1=\{\xi<\omega_1\st M_\xi\in Z\}.\]
Note that $Z$ is club in $[H_\theta^\V]^\omega$, since $\T_S \times \P_X$ is countably distributive. So $C_1$ is club in $\omega_1$ as well. 

By assumption, $\bb{T}_S\times\P_X$ is stationary set preserving, so $T$ is stationary in $\V[I][\vM]$. Hence, the set
\[T'=\{\xi<\omega_1\st M_\xi\cap\omega_1\in T\}\]
is also stationary. 

So let $\xi\in C_0 \cap C_1\cap T'$. Then $M_\xi\in X$, as $\vM$ is $\P_X$-generic, and $M_\xi\cap\omega_1\in T$, as $\xi\in T'$. Since $\xi\in C_1$, $M_\xi\in Z$. Thus, $M_\xi\in X_T\cap Z$. By choice of $Z$, this means that $M_\xi\in\lifting(S,[H_\theta]^\omega)$, that is, $\sup(M_\xi\cap\lambda)\in S$. But $\xi\in C_0$, so $\sup(M_\xi\cap\lambda)\in B$, and $B\cap S=\leer$. This is a contradiction.
\end{proof}

Recall that in Larson's construction of a model in which \SRP holds yet $\OSR_{\omega_2}$ fails, Theorem \ref{thm:AbstractPreservation} was instrumental: starting in a model of $\MM$, we forced with a poset of the form $\P$ to add a generic $G$ which constitutes a counterexample to diagonal reflection, and argued that if $S$ is a projective stationary set in the forcing extension, then we can find a further forcing $\dot{\T}$ (a club shooting forcing of the form $\T_S$) so that $\P*\dot{\T}$ has a dense $\omega_2$-directed closed subset and, letting $\P_X=\dot{\Q}^G$, $\P*(\dot{\Q}\times\dot{\T})$ is stationary set preserving. The theorem then implies that in $\V[G]$, $\FA_{\omega_1}(\{\P_X\})$ still holds. Now, $\P*(\dot{\Q}\times\dot{\T})$ is stationary set preserving iff in $\V[G]$, letting $\T_S=\dot{\T}^G$, $\P_X\times\T_S$ is stationary set preserving, since $\P$ does not add new subsets to $\omega_1$. And to show that $\P_X\times\T_S$ is stationary set preserving in $\V[G]$, we used the fact that this is equivalent to saying that in $\V[G]$, $X\ohne\lifting(S,[H_\theta]^\omega)$ is projective stationary, where $X$ is projective stationary in $[H_\theta]^\omega$ in $\V[G]$. Thus, the previous lemma was crucial. It turns out that this crucial equivalence carries over nicely to the setting of $\infty$-subcomplete forcings.

\begin{thm}
\label{thm:WhatWeGet}
Let $\lambda$ be a regular cardinal greater than $\omega_1$, and let $S\sub\lambda$ be a stationary set such that $\lambda \ohne S$ is unbounded in $\lambda$ and $\T_S$ is countably distributive. Let $\theta\ge\lambda$ be a cardinal, and let $X\sub[H_\theta]^\omega$ be spread out. 
Then the following are equivalent:
\begin{enumerate}[label=\textnormal{(\arabic*)}]
\item
\label{item:SpreadOutOffTheLiftingOfA}
$X\ohne\lifting(S,[H_\theta]^\omega)$ is spread out.	
\item
\label{item:CompositionSC}
$\T_S\times\P_X$ is $\infty$-subcomplete.
\end{enumerate}
\end{thm}

\begin{proof}

\noindent\ref{item:SpreadOutOffTheLiftingOfA}$\implies$\ref{item:CompositionSC}: Suppose not. Let $\sigma:\bcalN\prec \calN=\calL_\tau^A$ be as in Definition \ref{def:(infinity-)subcompleteness}, with $\bG\times\bH$ generic for $\bar{\T}_{\bS}\times\bar{\P}_{\bar{X}}$, the preimage of $\T_S\times\P_X$ under $\sigma$, over $\bcalN$.

Note that since $\T_S$ is countably distributive, it does not add any conditions to $\P_X$ (which consists of countable sequences from $S$), and $\T_S\times\P_X$ is forcing equivalent to $\T_S *\check{\P}_X$. Note also that by Lemma \ref{lem:WhenT_APreservesGenStationarity}, $X$ is stationary in $\V^{\T_S}$. So $\T_S * \check{\P}_X$ is countably distributive.

Fix some $\bar{a}\in\bcalN$. We are tasked with finding a condition $\kla{c,p}\in\T_S\times\P_X$ that forces the existence of a $\sigma':\bcalN\prec \calN$ with $\sigma'(\bar{a})=\sigma(a)$ such that $\sigma'``\bG\times\bar{H}\sub\dot{G}\times\dot{H}$, where $\dot{G}$, $\dot{H}$ are the canonical names for the $\T_S$, $\P_X$-generic filters. Since $\T_S \times \P_X$ is countably distributive, this is equivalent to showing that there is such a $\sigma'$ in $\V$.

Let $K=\bigcup X=H_\theta$.
By \ref{item:SpreadOutOffTheLiftingOfA}, there is an embedding $\sigma':\bcalN\prec \calN$ with $\sigma'(\ba)=\sigma(\ba)$ such that $\ran(\sigma')\cap K\in X$ and $\alpha=\sup(\ran(\sigma')\cap\lambda)\in\lambda\ohne S$. Let $c=(\bigcup\sigma'``\bG)\cup\{\alpha\}$. Then $c\in\T_S$, since $\alpha\notin S$. Let $p=(\bigcup\sigma'``\bar{H})\verl\kla{\sigma'``\bar{K}}$, where $\bar{K}=\sigma^{-1}(K)$ (note that by genericity, $\bar{K}=\bigcup\bigcup\{\ran(s)\st s\in\bar{H}\}$). But $\sigma'``\bar{K}=\ran(\sigma')\cap K \in X$, which means, again, that $p\in\P_X$. But the condition $\kla{c,p}\in\T_S\times\P_X$ is then as desired: if $G\times H$ is a $\T_S\times\P_X$-generic filter containing it, then $\sigma'``\bar{G}\times\bar{H}\sub G\times H$.

\ref{item:CompositionSC}$\implies$\ref{item:SpreadOutOffTheLiftingOfA}: Let $\calN$, $M$, $a$ be as in the definition of being spread out, $\calN=\calL_\tau^A$. Let $\sigma:\bcalN\prec \calN$, $M=\ran(\sigma)$. Let $\ba=\sigma^{-1}(a)$. Also, let $K=H_\theta$. We have to find a $\sigma':\bcalN\prec \calN$ with $\sigma'(\ba)=a$ and $\ran(\sigma')\cap K \in X \ohne \lifting(S,[K]^\omega)$.

Denoting preimages under $\sigma$ by adding bars, let $\bG\times\bH$ be $\bar{\T}_{\bar{S}}\times\bar{\P}_{\bar{X}}$-generic over $\bcalN$. By $\infty$-subcompleteness and countable distributivity of $\T_S \times \P_X$, let $\kla{c,p}\in\T_S\times\P_X$, $\sigma'$ be such that whenever $G\times H$ is $\T_S\times\P_X$-generic over $\V$ with $\kla{c,p}\in G\times H$, it follows that $\sigma'``\bar{G}\times\bar{H}\sub G\times H$. Since $\sigma'``\bG\sub G$ in this situation, it must be that $\tilde{c}=\bigcup\sigma'``\bG$, a set closed below its supremum, but having no maximum, can be extended by a condition in $\T_S$ (even in $G$). The weakest possible condition to do this is $\tilde{c}\cup\{\sup\tilde{c}\}$. But by genericity of $\bar{G}$, it follows that $\alpha:=\sup\tilde{c}=\sup ( \sigma'``\blambda )=\sup ( \lambda\cap\ran(\sigma') )$. So since $\tilde{c}\cup\{\alpha\}$ must be in $\T_S$, it follows that $\alpha\notin S$. Similarly, the sequence $\tilde{s}=\bigcup\sigma'``\bar{H}$ has no last element, but since it is contained in $H$ and has countable length, its extension by one element must be in $H$, and hence in $\P_X$. There is only one possible extension of it, namely it has to be extended by the union of the elements in the sequence. But by genericity of $\bH$, that union must be $\sigma'``\bar{K}$, which is the same as $\ran(\sigma')\cap K$. So since $\tilde{s}\verl(\ran(\sigma')\cap K)$ is a condition in $\P_X$, it follows that $\ran(\sigma')\cap K\in X$. Thus, $\ran(\sigma')\cap K \in X\ohne\lifting(S,[K]^\omega)$. This shows that the latter set is spread out.
\end{proof}

The key to Larson's argument showing that $\SRP$ does not imply $\OSR_{\omega_2}$ was that after adding a counterexample $\omega_2$-sequence $\vS$ of stationary subsets of $S^{\omega_2}_\omega$, it is the case that for any projective stationary set $X$, there is some $\alpha<\omega_2$ such that $X\ohne\lifting(S_\alpha,[\bigcup X]^\omega)$ is projective stationary, which means that $\T_{S_\alpha}\times\P_X$ is stationary set preserving. Similarly, for the argument that $\SRP$ does not imply $\uDSR(1,S^\lambda_\omega)$, where $\lambda>\omega_2$ is regular (or even $\lambda=\omega_2$, using a stronger assumption), the point was that there is an $\alpha<\lambda$ such that $X\ohne\lifting(S_{{\ge}\alpha},[\bigcup X]^\omega)$ is projective stationary, which again means that $\T_{S_{{\ge}\alpha}}\times\P_X$ is stationary set preserving. 

The original argument must break down when forcing over a model of $\infty$-$\SCFA+\CH$, say, because we know that $\infty$-\SC-\SRP $+$ \CH implies $\OSR_{\omega_2}$. But if we were to try to imitate this kind of reasoning in order to separate \infSC-\SRP from $\uDSR(1,S^\lambda_\omega)$ for some $\lambda>\omega_2$, we would be led to argue that after adding $\vS$, in $\V[\vS]$, for any spread out $X$, there is an $\alpha<\lambda$ such that $X\ohne\lifting(S_{{\ge}\alpha},[\bigcup X]^\omega)$ is spread out, which by Theorem \ref{thm:WhatWeGet} means that $\T_{S_{{\ge}\alpha}}\times\P_X$ is \infSC. It should be pointed out here that saying that $\P*(\dot{\Q}\times\dot{\mathbb{T}})$ is $\infSC$ is probably not equivalent to saying that in $\V[\vS]$, $\Q\times\T_S$ is $\infSC$ - the latter is potentially a stronger property. Regardless, the following observation shows some problems with the approach of trying to show that this potentially stronger property holds.

\begin{obs}
Assume \MM, and let $\P$ be the forcing to add a sequence $\seq{S_\alpha}{\alpha<\omega_3}$ such that $\vS$ witnesses the failure of $\uDSR(1,S^{\omega_3}_\omega)$ (see Theorem \ref{thm:EnhancedSRPDoesntImplyuDSR}). Then, in $\V[\vS]$, the following statements are true:
\begin{enumerate}[label=\textnormal{(\arabic*)}]
\item 
\label{item:SRP}
$\SRP^\FA$ holds, hence so does $\infty$-\SC-$\SRP^\FA$. 
\item 
\label{item:SRP-Strategy}
for every set $X$ spread out in $[H_\theta]^\omega$ for some regular $\theta\ge\omega_3$, there is an $\alpha<\omega_3$ such that $X\ohne\lifting(S_{\ge\alpha},[H_\theta]^\omega)$ is projective stationary.
\item
\label{item:infSCSRP-StrategyFails} 
there is a set $X$ spread out in $[H_{\omega_3}]^\omega$ for which there is no $\alpha<\omega_3$ such that $X\ohne\lifting(S_\alpha,[H_{\omega_3}]^\omega)$ is spread out.
\end{enumerate}
\end{obs}

\begin{proof} \ref{item:SRP} follows from the statement of Theorem \ref{thm:EnhancedSRPDoesntImplyuDSR}, and \ref{item:SRP-Strategy} is what the proof of that Theorem shows.

To see \ref{item:infSCSRP-StrategyFails}, working in $\V[\vS]$, fix a sequence $\vT = \seq{T_\alpha}{\alpha<\omega_3}$ of stationary subsets of $S^{\omega_2}_\omega$ such that if $\alpha<\beta<\omega_3$, then $T_\alpha\cap T_\beta$ is not stationary in $\omega_2$. This is possible, because the nonstationary ideal on $\omega_2$, restricted to $S^{\omega_2}_\omega$, is not $\omega_3$-saturated. As in Definition \ref{def:MagicalSets}, let $X\sub[H_{\omega_3}]^\omega$  be defined to consist of those $M$ such that $\kla{M,{\in}}\prec\kla{H_{\omega_3},{\in}}$ and such that either there is no $\alpha\in M\cap\omega_3$ with $\sup(M\cap\omega_2)\in T_\alpha$, or there is a (unique) such $\alpha$, and in this case, $\sup(M\cap\omega_3)\in S_\alpha$. It was shown in Lemma \ref{lem:MagicalSetsSpreadOutAboveKappa} that $X$ is spread out (above $\omega_2$); note here that $\omega_2=2^\omega$, so $(\omega_2)^\omega=\omega_2<\omega_3$.

To prove the crucial property of $X$, fix $\alpha<\omega_3$. We show that $X^-=X\ohne\lifting(S_{\alpha},[H_{\omega_3}]^\omega)$ is not spread out. Namely, let $\theta$ be some sufficiently large cardinal, and let $H_\theta\sub \calN=\calL_\tau^A\models\ZFCm$. We can always find a countable $M$ such that $\calN | M \prec \calN$ is full, $\alpha\in M$ and $\sup(M\cap\omega_2)\in T_\alpha$. Assume $X^-$ were spread out. Then there would be an $M'$ such that $\calN | M' \prec \calN$, and there is an isomorphism $\pi:\calN|M \To \calN|M'$ that fixes $\vS , \vT$ and $\alpha$, and $M'\cap H_{\omega_3}\in X^-$. But since $2^\omega=\omega_2$, $M\cap\omega_2=M'\cap\omega_2$, as $\power(\omega)\cap M=\power(\omega)\cap M'$ and every ordinal less than $\omega_2$ can be coded by a real; see \cite[Obs.~2.32]{Fuchs:CanonicalFragmentsOfSRP}. Since $\alpha$ is fixed by $\pi$, so $\alpha\in M'$, and we have $\sup(M'\cap\omega_2)=\sup(M\cap\omega_2)\in T_\alpha$. Then, since $M'\cap H_{\omega_3}\in X$, we have that $\sup(M'\cap\omega_3)\in S_\alpha$, so $M'\notin X^-$, a contradiction.
\end{proof}

This shows two problems: first, in $\V[\vS]$, $\infty$-\SC-\SRP holds, but the strategy \ref{item:SRP-Strategy} to prove this that works for \SRP does not work to show that $\infty$-\SC-\SRP holds - this is the import of \ref{item:infSCSRP-StrategyFails}. It also shows that projective stationary sets are more manageable than spread out ones. It must be added, though, that the situation would be different when forcing over a model of $\infSCFA+\CH$; in particular, the argument for (3) would not work.

So even though Larson's characterization of the preservation of projective stationarity carries over nicely to the context of spread out sets, we have not been able so far to separate $\infSC$-$\SRP+\CH$ from $\DSR(\omega_1,S^\lambda_\omega)$, for any regular $\lambda>\omega_1$. We have seen that this is impossible for $\lambda=\omega_2$. For $\lambda>\omega_2$, another way to appreciate the difficulty is as follows: if one starts in a model of \MM and adds a counterexample to $\uDSR(1,S^\lambda_\omega)$ as we did, then in the forcing extension, the enhanced version of $\SRP^\FA$ holds. The situation must change when starting in a model of $\infty$-$\SCFA + \diamondsuit$, say, because after adding the counterexample to $\uDSR(1,S^\lambda_\omega)$, it cannot be that the enhanced version of $\infSC$-$\SRP^\FA$ holds, as $\diamondsuit$ still holds, and these two principles taken together would imply $\DSR(\omega_1,S^\lambda_\omega)$. So either, adding the counterexample does not preserve $\infty$-\SC-\SRP (so the separation fails), contrary to the situation with \MM/\SRP, or it is preserved, but the enhanced version isn't, again, contrary to the situation with \MM/\SRP (so the separation works).

\begin{question}
Does $\infSC$-$\SRP+\CH$ imply (strongly exact) $\DSR(\lambda,S^\lambda_\omega)$ for any/all regular $\lambda>\omega_2$? Does it help to assume $\diamondsuit$? Or $\diamondsuit^+$?	
\end{question}

More generally, the following is still open.

\begin{question}
Does $\infSC$-$\SRP + \CH$ imply \infSCFA? 	
\end{question}

See \cite{Fuchs:CanonicalFragmentsOfSRP} for partial answers in the negative.

\bibliography{LittleLibrary}
\bibliographystyle{plain}
\end{document}